\documentclass[11pt]{amsart}

\usepackage[utf8]{inputenc}
\usepackage{mathtools}
\usepackage{amsmath}
\usepackage{amsfonts}
\usepackage{fancyhdr}
\usepackage{amssymb}
\usepackage{amsthm}
\usepackage[margin= 1 in]{geometry}

\usepackage[colorlinks=true, pdfstartview=FitV, linkcolor=blue, citecolor=blue, urlcolor=blue]{hyperref}

\usepackage{xcolor}
\usepackage{tikz}
\usetikzlibrary{cd}
\usepackage{ytableau}
\usepackage{makecell}
\usepackage{cleveref}
\usepackage{bm}

\usepackage[colorinlistoftodos]{todonotes}

\newcommand{\defn}[1]{{\color{darkred}\emph{#1}}} 
\definecolor{darkblue}{rgb}{0.0,0,0.7} 
\definecolor{darkred}{rgb}{0.7,0,0} 
\definecolor{darkgreen}{rgb}{0, .6, 0} 

\newcommand{\ZZ}{\mathbb{Z}}
\newcommand\BFT{\mathsf{BFT}}
\newcommand\EXQ{\mathsf{EXQ}}

\newcommand\ml{{\mu/\lambda}}

\newcommand{\HVT}{\mathsf{HVT}}

\newcommand{\FMT}{\mathsf{FMT}}

\newcommand{\SSYT}{\mathsf{SSYT}}

\newcommand{\shape}{\mathsf{shape}}
\newcommand{\wt}{\mathsf{wt}}

\newcommand{\U}{\mathcal{U}}
\newcommand{\A}{\mathcal{A}}
\renewcommand{\L}{\mathcal{L}}

\newcommand{\Abump}{\mathcal{A}_{b}}
\newcommand{\Lbump}{\mathcal{L}_{b}}
\newcommand{\shuff}{\mathsf{shuff}}

\newcommand{\jdt}{\mathsf{jdt}_{\mathrm{GG}}}

\newtheorem{theorem}{Theorem}[section]
\newtheorem{lemma}[theorem]{Lemma}
\newtheorem{corollary}[theorem]{Corollary}

\newtheorem{proposition}[theorem]{Proposition}

\theoremstyle{definition}
\newtheorem{definition}[theorem]{Definition}
\newtheorem{example}[theorem]{Example}
\newtheorem{remark}[theorem]{Remark}
\newtheorem{problem}[theorem]{Problem}
\numberwithin{equation}{section}

\title{Hook-valued tableau uncrowding and tableau switching}

\author[J.~Jang]{Jihyeug Jang}
\address[J. Jang]{Department of Mathematics,
Sungkyunkwan University (SKKU), Suwon, Gyeonggi-do 16419, South Korea}
\email{jihyeugjang@gmail.com}

\author[J.~S.~Kim]{Jang Soo Kim}
\address[J.~S.~Kim]{Department of Mathematics,
Sungkyunkwan University (SKKU), Suwon, Gyeonggi-do 16419, South Korea}
\email{jangsookim@skku.edu}

\author[J.~Pan]{Jianping Pan}
\address[J. Pan]{Department of Mathematics, North Carolina State University, Raleigh, NC 27695-8205, U.S.A.\\
Current address: School of Mathematical and Statistical Sciences, Arizona State University,
Tempe, AZ 85287-1804, U.S.A.}
\email{jianping.pan@asu.edu}

\author[J.~Pappe]{Joseph Pappe}
\address[J. Pappe]{Department of Mathematics, Colorado State University, 1874 Campus Delivery, Fort Collins, CO 80523-1874,U.S.A.}
\email{joseph.pappe@colostate.edu}

\author[A.~Schilling]{Anne Schilling}
\address[A. Schilling]{Department of Mathematics, UC Davis, One Shields Ave., Davis, CA 95616-8633, U.S.A.}
\email{aschilling@ucdavis.edu}

\date{\today}

\date{\today}

\begin{document}

\begin{abstract}
  Refined canonical stable Grothendieck polynomials were introduced by
  Hwang, Jang, Kim, Song, and Song. There exist two combinatorial
  models for these polynomials: one using hook-valued tableaux and
  the other using pairs of a semistandard Young tableau and (what we
  call) an exquisite tableau. An uncrowding algorithm on hook-valued
  tableaux was introduced by Pan, Pappe, Poh, and Schilling. In this
  paper, we discover a novel connection between the two models via the
  uncrowding and Goulden--Greene's jeu de taquin algorithms, using a classical result of
  Benkart, Sottile, and Stroomer on tableau switching. This
  connection reveals a symmetry of the uncrowding algorithm
  defined on hook-valued tableaux. As a corollary, we obtain another
  combinatorial model for the refined canonical stable Grothendieck
  polynomials in terms of biflagged tableaux, which naturally appear 
  in the characterization of the image of the uncrowding map.
\end{abstract}

\maketitle 

\section{Introduction}
\label{section.intro}

Refined canonical stable Grothendieck polynomials were introduced in~\cite{HJKSS1}, generalizing and unifying many of the previous variants
of Grothendieck polynomials. They encompass the symmetric Grothendieck polynomials introduced by Lascoux and Sch\"utzenberger~\cite{LS.1982}, 
the symmetric $\beta$-Grothendieck polynomials of Fomin and Kirillov~\cite{FK.1996}, 
the canonical stable Grothendieck polynomials of 
Yeliussizov~\cite{Yeliussizov.2017}, 
and the refined stable Grothendieck polynomials of Chan and Pflueger~\cite{CP.2021}.
The refined canonical stable Grothendieck polynomials were further studied by Iwao, Motegi, and Scrimshaw \cite{Iwao2024}
in terms of vertex operators.
We note that dual Grothendieck polynomials and generalizations have also been studied;
 see \cite{LP.2007, GGL2016, HJKSS1} and references therein.

In analogy to Schur functions, which are generating functions of semistandard Young tableaux, Buch~\cite{Buch.2002} showed that 
symmetric Grothendieck polynomials are generating functions of set-valued tableaux. Each cell in a set-valued tableau contains
a set instead of just a natural number. Combinatorial models for the (refined) canonical stable Grothendieck polynomials are described in terms of hook-valued 
tableaux~\cite{Yeliussizov.2017,HJKSS2}, which contain a semistandard Young tableau of a hook shape in each cell. These tableau
models are intimately related to the monomial expansions of the different versions of the Grothendieck polynomials.

An important question is to find the Schur expansion of the various symmetric Grothendieck polynomials.
Lenart~\cite{Lenart.2000} gave the Schur expansion of the symmetric Grothendieck polynomials, whose monomial expansion
is given in terms of set-valued tableaux. Buch~\cite[Theorem 6.11]{Buch.2002} developed an uncrowding algorithm on set-valued
tableaux to give a bijective proof of Lenart's Schur expansion.
The uncrowding algorithm on a set-valued tableau produces a pair of a semistandard Young tableau and a tableau with certain flagged conditions,
using the RSK bumping algorithm to uncrowd cells that contain more than one integer; see \cite{LP.2007, BM.2012, Pat.2016}.
This uncrowding algorithm was used by Reiner, Tenner and Yong~\cite{RTY.2018}, and generalized by Chan and Pflueger~\cite{CP.2021}
and by Pan, Pappe, Poh, and Schilling~\cite{PPPS}.

Hwang et al. \cite{HJKSS1} found the Schur expansion for refined canonical Grothendieck polynomials,
which we rephrase in terms of ``exquisite'' tableaux. 
Hence, it is a natural question to relate the combinatorial model for the refined canonical Grothendieck polynomials in terms of hook-valued tableaux
with the combinatorial model in terms of exquisite tableaux by giving a bijection between hook-valued tableaux and pairs of a semistandard tableau and
an exquisite tableau. In this paper, we find such a bijection (see Theorem~\ref{thm:2bij}) using two types of uncrowding maps by combining the 
uncrowding algorithm due to Pan et al.~\cite{PPPS} with Goulden--Greene's jeu de taquin~\cite{GG}.

The uncrowding algorithm on hook-valued tableaux in~\cite{PPPS} uncrowds
the entries in the arms of the hooks and yields a set-valued tableau and a tableau with certain flagged conditions.
Subsequently applying the uncrowding
algorithm by Buch~\cite{Buch.2002} on the set-valued tableau yields a semistandard Young tableau and a recording tableau. It was proved
in~\cite{PPPS} that this uncrowding operator intertwines with the crystal operators of Hawkes and Scrimshaw~\cite{HS}.
Let us denote this sequence of uncrowding operations by $\U_{\L^\infty \A^\infty}$, which indicates that first arm and then leg uncrowding is performed.
In this paper, we also consider other orderings of leg and arm uncrowding, in particular $\U_{\A^\infty \L^\infty}$ which first performs leg and then arm
uncrowding. We relate the two orderings using tableau switching in the sense of Benkart, Sottile and Stroomer~\cite{BSS}.
This yields a characterization of the recording tableaux under the uncrowding algorithm in terms of biflagged tableaux (see Corollary~\ref{corollary.image}).
To connect to the combinatorial model of  \cite{HJKSS1} in terms of exquisite tableaux we use the jeu de taquin algorithm due to Goulden and 
Greene~\cite{GG}, which we call GG-jdt. This map
was further studied by Krattenthaler \cite{Kra}.
We show that  GG-jdt is a partial tableau switching procedure (see Proposition~\ref{prop.perforated}).
This yields the bijection between the combinatorial models for the refined canonical stable Grothendieck polynomials (see Theorem~\ref{thm:2bij}).
Our proof reveals a symmetry of the uncrowding algorithm on hook-valued tableaux
when interchanging the order of arm and leg uncrowding (see Theorem~\ref{thm.shuff.q}). A corollary of these results is the equivalence of
three combinatorial models for the refined canonical stable Grothendieck polynomials (see Corollary~\ref{corollary.grothendieck}).

The various tableaux used in this paper are summarized in Table~\ref{table.tableaux}.
\begin{table}
\begin{tabular}{|l|c|l|}
\hline
Tableau & Symbol & Appearance\\
\hline \hline
hook-valued tableaux & \(\HVT\) & \Cref{def:HVT}\\
mixed tableaux & n/a & \Cref{def:MT}\\
flagged-mixed tableaux & \(\FMT\) & \Cref{def:FMT}\\
exquisite tableaux & \(\EXQ\) & \Cref{def:exq}\\
biflagged tableaux & \(\BFT\) & \Cref{def:BFT}\\
\hline
\end{tabular}
\caption{Various tableaux used in this paper.}
\label{table.tableaux}
\end{table}

The paper is organized as follows. In Section~\ref{section.background}, we review the refined canonical Grothendieck polynomials
and exquisite tableaux. In Section~\ref{section.uncrowding}, we review and generalize the uncrowding algorithm of~\cite{PPPS}
and introduce the tableau switching procedure of~\cite{BSS} in our notation.
In Section~\ref{section.proof}, we present our results and proofs.
In Section~\ref{subsec.armleg}, we show that interchanging arm and leg uncrowding can
be interpreted in terms of tableau switching on the recording tableaux. 
In Section~\ref{subsec.ggjdt}, we show that GG-jdt can be formulated as a partial tableau switching procedure.
In Section~\ref{subsec.image}, we present a bijection between hook-valued tableaux
and pairs of a semistandard Young tableau and an exquisite tableau.
In doing so, we also show that GG-jdt is a bijection between biflagged tableaux and exquisite tableaux.
As corollaries, we obtain the equivalence of three combinatorial models for refined Grothendieck polynomials
and a characterization of the image of the uncrowding algorithm.
Finally, in \Cref{sec:further-study}, we propose some open problems.

\section*{Acknowledgments}
The authors acknowledge the inspiring atmosphere at FPSAC 2023 in Davis, where this project started. 
J. Pan would like to thank the Association for Women in Mathematics for providing support to travel to FPSAC 2023. 
The authors also thank Maria Gillespie for discussions on tableau switching
and Travis Scrimshaw for helpful comments.
They would also like to thank the anonymous referee for careful reading and useful comments.
The workshop ``Advances in Lie Theory, Representation
Theory, and Combinatorics'' held at SLMath in May 2024 inspired by the work of Georgia Benkart gave us the idea to use tableau
switching~\cite{BSS}.

J.~S.~Kim was supported by the National Research Foundation of Korea (NRF) grant funded by the Korea government RS-2025-00557835.
A. Schilling was partially supported by NSF grant DMS--2053350.

\section{Preliminaries}
\label{section.background}

In this section, we review the refined canonical stable Grothendieck polynomials and exquisite tableaux~\cite{HJKSS1}.
We use French notation for tableaux throughout the paper. We
refer the reader to \cite[Chapter~7]{EC2} for basic notions such as
partitions, Young diagrams, skew shapes, semistandard Young tableaux,
and Schur functions \( s_\lambda(\bm{x}) \), where
$\bm{x} = (x_1,x_2,\dots)$ is an infinite set of indeterminants.

\subsection{Refined canonical stable Grothendieck polynomials}
\label{subsec.grothendieck}

Yeliussizov~\cite{Yeliussizov.2017} introduced the \defn{canonical stable Grothendieck polynomial} $G^{(\alpha,\beta)}_\lambda(\bm{x})$
indexed by a partition $\lambda$ and two parameters $\alpha$ and $\beta$. It is a formal power series generalizing the 
\defn{symmetric \( \beta \)-Grothendieck polynomial} $G^{(\beta)}_\lambda(\bm{x})$ with the property
\[
\omega(G^{(\alpha,\beta)}_\lambda(\bm{x})) = G^{(\beta,\alpha)}_{\lambda'}
(\bm{x}),
\]
where $\omega$ is the involution that sends the Schur function
$s_{\lambda}(\bm{x})$ to $s_{\lambda'}(\bm{x})$ indexed by the transpose partition $\lambda'$. 

Hwang, Jang, Kim, Song, and Song \cite{HJKSS1} introduced the
\defn{refined canonical stable Grothendieck polynomials}
$G_\lambda(\bm{x}, \bm{\alpha},\bm{\beta})$, which are refinements of
$G^{(\alpha,\beta)}_\lambda(\bm{x})$ with infinite sets of parameters
\( \bm{\alpha}= (\alpha_1,\alpha_2,\dots) \) and
\( \bm{\beta}= (\beta_1,\beta_2,\dots) \). They are defined by
$G_\lambda(\bm{x}, \bm{\alpha},\bm{\beta}) = \lim_{n\to
  \infty}G_\lambda(\bm{x}_n, \bm{\alpha},\bm{\beta})$,
where $\bm{x}_n = (x_1,\dots,x_n)$ and
\[
  G_\lambda(\bm{x}_n, \bm{\alpha},\bm{\beta}) 
  = \dfrac{\det\left(\frac{(1+\beta_1x_j)\cdots (1+\beta_{i-1}x_j)}{(1-\alpha_1x_j)\cdots 
  (1-\alpha_{\lambda_i}x_j)} \right)_{1\leqslant i,j\leqslant n}}{\prod_{1\leqslant i < j \leqslant n}(x_i-x_j)}.
\]
Here, we replaced every \( \beta_i \) with \( -\beta_i \) in the
original definition \cite[Definition~1.1]{HJKSS1} in order to make
\( G_\lambda(\bm{x}, \bm{\alpha},\bm{\beta}) \) a formal power
series with positive coefficients. We
can set $\alpha_i = \alpha$ and $\beta_i = \beta$, for all $i$, in the
refined version $G_\lambda(\bm{x}, \bm{\alpha},\bm{\beta})$ to get the
original version $G^{(\alpha,\beta)}_\lambda(\bm{x})$.

Combinatorially, both $G^{(\alpha,\beta)}_\lambda(\bm{x})$ and
$G_\lambda(\bm{x}, \bm{\alpha},\bm{\beta})$ are the generating
functions for hook-valued tableaux~\cite{Yeliussizov.2017,HJKSS2}. 

\begin{definition}\label{def:HVT}
A \defn{hook-valued tableau} is a filling of a partition shape satisfying the following conditions:
\begin{enumerate}
  \item Each box contains a semistandard Young tableau of hook shape, i.e.,
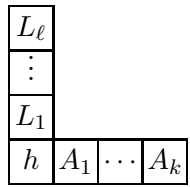
  \[\begin{ytableau}
    L_\ell\\
    \vdots\\
    L_1\\
    h & A_1 & \cdots & A_k
  \end{ytableau}\]
  where $h<L_1<\dots<L_\ell$ and $h\leqslant A_1 \leqslant \dots \leqslant A_k$ are positive integers. The entry $h$ is called the \defn{hook entry}, the $L_i$'s the \defn{leg entries} and the $A_i$'s the \defn{arm entries}. 
  \item Each row is weakly increasing, i.e., any entry in a box is weakly smaller than any entry in the box directly to the right of it.
  \item Each column is strictly increasing, i.e., any entry in a box is strictly smaller than any entry in the box directly above it.
\end{enumerate}
We denote by $\HVT(\lambda)$ the set of hook-valued tableaux of shape
$\lambda$. We also write $\HVT$ for the set of hook-valued tableaux of
any shape. The \defn{weight} \( \wt(T) \) of a hook-valued tableau $T$
is defined by
\[
  \wt(T) = \prod_{i\geqslant 1}
  \alpha_i^{(\# \text{ of arm entries in column } i)}
  \beta_i^{(\# \text{ of leg entries in row } i)}
  x_i^{(\# \text{ of $i$'s in } T)}.
\]
\end{definition}

\begin{example}
  The tableau $T_1$ is a hook-valued tableau of shape $(3,2)$, and
  $T_2$ is not a hook-valued tableau because the first row of \( T_2 \) is not weakly
  increasing (the first column of \( T_2 \) is not strictly increasing either):
\[{\def\mc#1{\makecell[lb]{#1}}
T_1 =\,{\begin{array}[lb]{*{3}{|l}|}\cline{1-2}
\mc{6\\4\\335}&\mc{67}\\\cline{1-3}
\mc{2\\11}&\mc{4\\334}&\mc{9\\445}\\\cline{1-3}
\end{array}\,,\qquad}
T_2 =\,{\begin{array}[lb]{*{4}{|l}|}\cline{1-2}
\mc{7\\445}&\mc{789}\\\cline{1-4}
\mc{4\\3\\122}&\mc{5\\345}&\mc{8\\7\\567}&\mc{7}\\\cline{1-4}
\end{array}\,.}
}
\]
The weight of \( T_1 \) is
\[
\wt(T_1) = \alpha_1^3\alpha_2^3\alpha_3^2 \beta_1^3\beta_2^2 x_1^2x_2x_3^4x_4^5x_5^2x_6^2x_7x_9.
\]
\end{example}

\begin{remark}
When all arm entries are empty, a hook-valued tableau is also called a \defn{set-valued tableau}.
\end{remark}

\begin{theorem}\label{thm:HVT} \cite{HJKSS2}
  We have
\[
  G_\lambda(\bm{x}, \bm{\alpha},\bm{\beta}) = \sum_{T\in \HVT(\lambda)} \wt(T).
\]
\end{theorem}

\subsection{Various mixed tableaux}

In this subsection, we introduce basic definitions that will be used
throughout this paper.

\begin{definition}\label{def:MT}
A \defn{mixed tableau} of shape \( \ml \) is a filling \( T \) of the
cells of \( \ml \) with elements in
\( \{\alpha_k \mid k\in \ZZ_{>0}\}\cup \{\beta_k \mid k\in \ZZ\} \).
The \defn{weight} \( \wt(T) \) of a mixed tableau \( T \) is the
product of its entries.
\end{definition}

\begin{definition}\label{def:FMT}
A \defn{flagged-mixed tableau} is a mixed tableau satisfying the
following conditions:
  \begin{enumerate}
  \item If \( T(i,j) = \alpha_k \), then \( 0<k<j \).
  \item If \( T(i,j) = \beta_k \), then \( 0<k<i \).
  \end{enumerate}
  Here, $T(i,j)$ denotes the entry in the $i$\textsuperscript{th} row and $j$\textsuperscript{th} column.
  We denote by \( \FMT \) the set of all flagged-mixed tableaux.
\end{definition}

\begin{definition}\label{def:strict}
  Let \( T \) be a mixed tableau. For
  \( \gamma\in \{\alpha,\beta\} \), we say that \( T \) is
  \defn{\( \gamma \)-column-strict}
  (resp.~\defn{\( \gamma \)-row-strict}) if the following
  conditions hold:
  \begin{enumerate}
  \item If \(\gamma_i\) and \(\gamma_j\) are entries in \(T\) such that
    \( \gamma_i \) is weakly southwest of \( \gamma_j \), then
    \( i\geqslant j \).
  \item If \(\gamma_i\) and \(\gamma_j\) are entries in \(T\) in the same
    column (resp.~row), then \( i\ne j \).
\end{enumerate}

We also say that \( T \) is \defn{totally column-strict} if the following
conditions hold:
\begin{enumerate} 
\item All indices of \( \alpha \) and \( \beta \) in each column are strictly decreasing.
  More precisely, if \( T(i,j)=\gamma_r \) and
  \( T(i+1,j)=\delta_s \) with
  \( \gamma,\delta\in \{\alpha,\beta\} \), then \( r> s \).
\item All indices of \( \alpha \) and \( \beta \) in each row are
  weakly decreasing. More precisely, if \( T(i,j)=\gamma_r \) and
  \( T(i,j+1)=\delta_s \) with
  \( \gamma,\delta\in \{\alpha,\beta\} \), then \( r\geqslant s \).
\end{enumerate}
\end{definition}

\begin{definition}\label{def:sorted}
  Let \( T \) be a mixed tableau of shape \( \ml \). Let \( A \)
  (resp.~\( B \)) be the set of cells in \( T \) containing
  \( \alpha_k \) (resp.~\( \beta_k \)) for any \( k\in \ZZ \). We say
  that \( T \) is \defn{\( (\alpha,\beta) \)-sorted} if \( A \) and
  \( B \) form skew shapes \( \nu/\lambda \) and \( \mu/\nu \),
  respectively, for some partition \( \nu \) with
  \( \lambda\subseteq\nu\subseteq\mu \).
  Similarly, we say that \( T \) is \defn{\( (\beta,\alpha) \)-sorted} if \( B \) and
  \( A \) form skew shapes \( \nu/\lambda \) and \( \mu/\nu \),
  respectively, for some partition \( \nu \) with
  \( \lambda\subseteq\nu\subseteq\mu \).
\end{definition}

\subsection{Exquisite tableaux}
\label{subsec.exquisite}

Hwang et al.~\cite{HJKSS1} provided a weighted
combinatorial model for the Schur expansion of
$G_\lambda(\bm{x}, \bm{\alpha},\bm{\beta})$ in terms of
$\ZZ$-inelegant tableaux. We rephrase their Schur expansion formula
using exquisite tableaux.

The \defn{content} \( c(i,j) \) of the cell \( (i,j) \) is defined by
$c(i,j) = j-i$.

\begin{definition}\label{def:c+-}
  Let \( T \) be a mixed tableau. We define \( c_\beta^+(T) \)
  (resp.~\( c_\beta^-(T) \)) to be the tableau obtained from \( T \)
  by replacing every \( \beta_r \) by \( \beta_{r+c} \)
  (resp.~\( \beta_{r-c} \)), where \( c \) is the content of the cell
  containing \( \beta_r \).
\end{definition}

\begin{definition}\label{def:exq}
 An \defn{exquisite tableau} is a flagged-mixed tableau $E$ such that
  \( c_\beta^+(E) \) totally column-strict. Let \(\EXQ(\ml)\) denote
  the set of all exquisite tableaux of shape \(\ml\).
\end{definition}

\begin{example}\label{exa:1}
Let \(\lambda = (2,1)\) and \(\mu = (3,3,1)\). Then
  \[
\EXQ(\ml) = \left\{
\raisebox{1em}{\begin{ytableau}
\beta_2 \\
*(gray) & \beta_1 & \alpha_1 \\
*(gray) & *(gray) & \alpha_2
\end{ytableau}}\,,
\raisebox{1em}{\begin{ytableau}
\beta_1 \\
*(gray) & \beta_1 & \alpha_1 \\
*(gray) & *(gray) & \alpha_2
\end{ytableau}}\,,
\raisebox{1em}{\begin{ytableau}
\beta_2 \\
*(gray) & \alpha_1 & \alpha_1 \\
*(gray) & *(gray) & \alpha_2
\end{ytableau}}\,,
\raisebox{1em}{\begin{ytableau}
\beta_1 \\
*(gray) & \alpha_1 & \alpha_1 \\
*(gray) & *(gray) & \alpha_2
\end{ytableau}}
\right\}\,.
\]
To see this, note that if \( T\in \EXQ(\ml) \), then \( T(1,3) \) must
be \( \alpha_r \) for some \( 0<r<3 \) because if
\( T(1,3)=\beta_k \), then \( 0<k<1 \), which is impossible. By the
condition that \( c_\beta^+(T) \) is totally column-strict, we must
have \( T(1,3)=\alpha_2 \). If \( T(2,3)=\beta_k \), then
\( k+c(2,3)=k+1<2 \), which is impossible. Hence,
\( T(2,3)=\alpha_1 \). Observe that
\( T(2,2)\in \{\alpha_1,\beta_1\} \) and
\( T(3,1)\in \{\beta_1,\beta_2\} \) can be chosen freely. Thus, there
are four exquisite tableaux as shown above.
\end{example}

Hwang et al.~\cite{HJKSS1} showed that if we
expand \( G_\lambda(\bm{x}, \bm{\alpha},\bm{\beta}) \) as a linear
combination of Schur functions, the coefficients are generating
functions for ``$\ZZ$-inelegant tableaux''.
We rephrase their result using exquisite tableaux as follows.

\begin{theorem}\cite[Corollary~4.5]{HJKSS1}
\label{thm.exq}
We have
\[
  G_\lambda(\bm{x}, \bm{\alpha},\bm{\beta}) = \sum_{\mu \supseteq \lambda} s_\mu(\bm{x}) \sum_{E \in \EXQ(\ml)} \wt(E).
\]
\end{theorem}

\section{Uncrowding algorithm and tableau switching}
\label{section.uncrowding}

In Section~\ref{subsec.uncrowding} we review and generalize the uncrowding algorithm in~\cite{PPPS}.
In Section~\ref{subsec.switching} we review the tableau switching procedure of~\cite{BSS}.

\subsection{Uncrowding algorithms for hook-valued tableaux}
\label{subsec.uncrowding}

In this subsection, we define the uncrowding algorithms for
hook-valued tableaux following~\cite{PPPS}. Note, however, that we use
a different convention for the recording tableau. We begin by defining
arm- and leg-uncrowding bumpings.

\begin{definition}
  The \defn{arm-uncrowding bumping} \( \A_b \colon \HVT\to \HVT \) is defined
  by the following algorithm:
  \begin{enumerate}
  \item Suppose \( T\in \HVT \) is given as an input.
  \item If \( T \) has no arms, then define \( \A_b(T) = T \). 
  \item From now on, suppose that \( T \) has arms. Find the rightmost
    column containing an arm. Let \( a \) be the largest arm entry in
    this column and let \( (r,c) \) be the cell containing it. If
    there are multiple arm entries equal to \( a \), then all of them
    must be in the same cell \( (r,c) \). In this case, we select the
    rightmost arm entry \( a \) in this cell.
  \item Find the smallest entry \( k \) in column \( c+1 \) with
    \( k\geqslant a \). There are two cases.   
   \begin{description}
   \item[Case 1] There exists such \( k \). Suppose that it is in
     \( (\widetilde{r},c+1) \). Remove the entry \( a \) in
     \( (r,c) \), replace the \( k \) in \( (\widetilde{r},c+1) \)
     with \( a \), and then add \( k \) as an arm in
     \( (\widetilde{r},c+1) \). If \( r=\widetilde{r} \), then in
     addition to this, we also remove all leg entries in \( (r,c) \) that are
     greater than \( a \) and add them to \( (r,c+1) \) as legs. 
   \item[Case 2] There is no such \( k \). Then create a new cell at
     the top of column \( c+1 \) and move the entry \( a \) in the
     cell \( (r,c) \) to the new cell. If \( (r,c+1) \) is the new
     cell, then we also remove all leg entries in \( (r,c) \) that are
     greater than \( a \) and add them to \( (r,c+1) \) as legs.
   \end{description}
 \item Finally, we define \( \A_b(T) \) to be the resulting tableau.
  \end{enumerate}
\end{definition}

\begin{example}\label{ex:Abump}
  The following shows hook-valued tableaux \( T, \Abump(T), \Abump^2(T), \) and
  \( \Abump^3(T) \):
\[{\def\mc#1{\makecell[lb]{#1}}
\,{\begin{array}[lb]{*{4}{|l}|}\cline{1-3}
\mc{2}&\mc{3\mathbf{4}}&\mc{9}\\\cline{1-4}
\mc{11}&\mc{2\\12}&\mc{6\\5\\3}&\mc{8\\7}\\\cline{1-4}
\end{array}\,},\quad
\,{\begin{array}[lb]{*{4}{|l}|}\cline{1-3}
\mc{2}&\mc{3}&\mc{9}\\\cline{1-4}
\mc{11}&\mc{2\\12}&\mc{6\\4\\3\mathbf{5}}&\mc{8\\7}\\\cline{1-4}
\end{array}\,},\quad
\,{\begin{array}[lb]{*{4}{|l}|}\cline{1-3}
\mc{2}&\mc{3}&\mc{9}\\\cline{1-4}
\mc{11}&\mc{2\\12}&\mc{4\\3}&\mc{8\\6\\5\mathbf{7}}\\\cline{1-4}
\end{array}\,},\quad
\,{\begin{array}[lb]{*{5}{|l}|}\cline{1-3}
\mc{2}&\mc{3}&\mc{9}\\\cline{1-5}
\mc{11}&\mc{2\\12}&\mc{4\\3}&\mc{\phantom{7}\\6\\5}&\mc{8\\7}\\\cline{1-5}
\end{array}\,}.
}
\]
For the first three tableaux, the largest arm entry in the rightmost column
containing an arm is indicated in boldface.
\end{example}

\begin{definition}
  The \defn{leg-uncrowding bumping} \( \L_b\colon \HVT\to \HVT \) is defined
  by the following algorithm.
  \begin{enumerate}
  \item Suppose \( T\in \HVT \) is given as an input.
  \item If \( T \) has no legs, then define \( \L_b(T) = T \). 
  \item From now on, suppose that \( T \) has legs. Find the topmost
    row containing a leg. Let \( \ell \) be the largest leg entry in this row
    and let \( (r,c) \) be the cell containing it.
  \item Find the smallest entry \( k \) in row \( r+1 \) with
    \( k> \ell \). There are two cases.   
    \begin{description}
   \item[Case 1] There exists such \( k \). Suppose that it is in
     \( (r+1,\widetilde{c}) \). Remove the entry \( \ell \) in
     \( (r,c) \), replace the \( k \) in \( (r+1,\widetilde{c}) \)
     with \( \ell \), and then add \( k \) as a leg in
     \( (r+1,\widetilde{c}) \). If \( c=\widetilde{c} \), then in
     addition to this, we also remove all arm entries in \( (r,c) \) that are
     at least \( \ell \) and add them to \( (r,c+1) \) as arms. 
   \item[Case 2] There is no such \( k \). Then create a new cell at
     the end of row \( r+1 \) and move the entry \( \ell \) in the
     cell \( (r,c) \) to the new cell. If \( (r+1,c) \) is the new
     cell, then we also remove all arm entries in \( (r,c) \) that are
     at least \( \ell \) and add them to \( (r+1,c) \) as arms.
   \end{description}
    \item Finally, we define \( \L_b(T) \) to be the resulting tableau.
  \end{enumerate}
\end{definition}

\begin{example}\label{ex:Lbump}
  The following shows hook-valued tableaux \( T, \Lbump(T) \), and
  \( \Lbump^2(T) \):
\[{\def\mc#1{\makecell[lb]{#1}}
\,{\begin{array}[lb]{*{3}{|l}|}\cline{1-1}
\mc{4}\\\cline{1-2}
\mc{33}&\mc{\phantom{1}\\56}\\\cline{1-3}
\mc{2\\1}&\mc{\mathbf{3}\\224}&\mc{45}\\\cline{1-3}
\end{array}\,},\quad
\,{\begin{array}[lb]{*{3}{|l}|}\cline{1-1}
\mc{4}\\\cline{1-2}
\mc{33}&\mc{\mathbf{5}\\346}\\\cline{1-3}
\mc{2\\1}&\mc{22}&\mc{45}\\\cline{1-3}
\end{array}\,},\quad
\,{\begin{array}[lb]{*{3}{|l}|}\cline{1-2}
\mc{4}&\mc{56}\\\cline{1-2}
\mc{33}&\mc{\phantom{5}\\34\phantom{6}}\\\cline{1-3}
\mc{2\\1}&\mc{22}&\mc{45}\\\cline{1-3}
\end{array}\,}.
}
\]
For the first two tableaux, the largest leg entry in the topmost row
containing a leg is written in boldface.
\end{example}

\begin{definition}
  The \defn{single-arm-uncrowding map} \( \A \colon \HVT\to \HVT \) is
  defined as follows. Let \( T\in \HVT \). If \( T \) has no arms, or
  equivalently, if \( \A_b(T) = T \), then define \( \A(T) = T \).
  Otherwise, we define \( \A(T) = \A_b^m(T) \), where \( m \) is the
  smallest integer such that the shape of \( \A_b^m(T) \) is larger
  than that of \( T \).
\end{definition}

\begin{definition}
  The \defn{single-leg-uncrowding map} \( \L \colon \HVT\to \HVT \) is
  defined as follows. Let \( T\in \HVT \). If \( T \) has no legs,
  then \( \L(T) = T \).
  Otherwise, we define \( \L(T) = \L_b^m(T) \), where \( m \) is the
  smallest integer such that the shape of \( \L_b^m(T) \) is larger
  than that of \( T \).
\end{definition}

For example, if \( T \) is the first tableau in \Cref{ex:Abump},
then \( \A(T) = \Abump^3(T) \), which is the last tableau
there. If \( T \) is the first tableau in \Cref{ex:Lbump}, then
\( \L(T) = \Lbump^2(T) \), which is the last tableau there.

We are now ready to define the uncrowding map.

\begin{definition}\label{defn.uncrowd}
  Let \( \U = \U_{f_n\cdots f_1} \), where \( f_n\cdots f_1 \) is a
  word in the alphabet \( \{\A,\L\} \).
  Then the \defn{uncrowding map}
  \( \U \colon \HVT\to \HVT\times \FMT \) is defined as follows.

  Let \( T\in \HVT \). We construct two tableaux \( P \) and
  \( Q \). For each \( i=0,1,2,\dots,n \), let
  \( T_i = f_i \circ \cdots \circ f_1 (T) \), where \( T_0=T \), and
  let \( \lambda^{(i)} \) be the shape of \( T_i \). First, we define
  \( P = T_n \). Now we define a flagged-mixed tableau \( Q \) of
  shape \( \lambda^{(n)}/\lambda^{(0)} \) as follows. For each
  \( i=1,2,\dots,n \), there are two cases.

  \begin{description}
  \item[Case 1] \( f_i = \A \). 
    In this case,
    \( T_{i} = \A(T_{i-1}) \). If \( T_{i-1} \) has arms, suppose
    that \( (r,c) \) is the cell that contains the largest arm entry in the
    rightmost column containing an arm. Then fill the unique cell
    \( \lambda^{(i)}/\lambda^{(i-1)} \) in \( Q \) with
    \( \alpha_c \). (If \( T_{i-1} \) has no arms, nothing happens in
    this case.)
  \item[Case 2] \( f_i = \L \). 
    In this case,
    \( T_{i} = \L(T_{i-1}) \). If \( T_{i-1} \) has legs, suppose
    that \( (r,c) \) is the cell that contains the largest leg entry in the
    topmost row containing a leg. Then fill the unique cell
    \( \lambda^{(i)}/\lambda^{(i-1)} \) in \( Q \) with \( \beta_r \).
    (If \( T_{i-1} \) has no legs, nothing happens in this case.)
  \end{description}
  
  Finally, we define \( \U(T) = (P,Q) \). We call \( P \) and \( Q \)
  the \defn{insertion tableau} and \defn{recording tableau} of
  \( \U(T) \), respectively. 
\end{definition}

\begin{example}\label{ex:uncrowding}
  Consider the following hook-valued tableaux
  \( T, \A(T), \A\circ\A(T), \L\circ\A\circ\A(T) \), and \( \L\circ \L\circ\A\circ\A(T) \):
\[{\def\mc#1{\makecell[lb]{#1}}
    \def\asep{\,\substack{\A\\ \rightarrow}\,\,}
    \def\lsep{\,\substack{\L\\ \rightarrow}\,\,}
\,{T=\begin{array}[lb]{*{4}{|l}|}\cline{1-1}
\mc{4}\\\cline{1-2}
\mc{3}&\mc{6\\57}\\\cline{1-2}
\mc{2}&\mc{24}\\\cline{1-4}
\mc{1}&\mc{1}&\mc{1}&\mc{5\\3}\\\cline{1-4}
\end{array}\,} \asep
\,{\begin{array}[lb]{*{4}{|l}|}\cline{1-1}
\mc{4}\\\cline{1-2}
\mc{3}&\mc{6\\5}\\\cline{1-3}
\mc{2}&\mc{24}&\mc{7}\\\cline{1-4}
\mc{1}&\mc{1}&\mc{1}&\mc{5\\3}\\\cline{1-4}
\end{array}\,} \asep
\,{\begin{array}[lb]{*{4}{|l}|}\cline{1-1}
\mc{4}\\\cline{1-2}
\mc{3}&\mc{6\\5}\\\cline{1-4}
\mc{2}&\mc{2}&\mc{4}&\mc{7}\\\cline{1-4}
\mc{1}&\mc{1}&\mc{1}&\mc{5\\3}\\\cline{1-4}
\end{array}\,} \lsep
\,{\begin{array}[lb]{*{4}{|l}|}\cline{1-2}
\mc{4}&\mc{6}\\\cline{1-2}
\mc{3}&\mc{\phantom{6}\\5}\\\cline{1-4}
\mc{2}&\mc{2}&\mc{4}&\mc{7}\\\cline{1-4}
\mc{1}&\mc{1}&\mc{1}&\mc{5\\3}\\\cline{1-4}
\end{array}\,} \lsep
\,{\begin{array}[lb]{*{4}{|l}|}\cline{1-2}
\mc{4}&\mc{6}\\\cline{1-3}
\mc{3}&\mc{\phantom{6}\\5}&\mc{7}\\\cline{1-4}
\mc{2}&\mc{2}&\mc{4}&\mc{5}\\\cline{1-4}
\mc{1}&\mc{1}&\mc{1}&\mc{\phantom{7}\\3}\\\cline{1-4}
\end{array}}
}\,.
\]
This shows that \( \U_{\L\L\A\A}(T) = (P,Q) \), where
\[
  P = \raisebox{2em}{
  \begin{ytableau}
4 & 6\\
3 & 5 & 7\\
2 & 2 & 4 & 5\\
1 & 1 & 1 & 3
\end{ytableau}\,}, \qquad 
Q = \raisebox{2em}{
\begin{ytableau}
*(gray) & \beta_3\\
*(gray) & *(gray) & \beta_1\\
*(gray) & *(gray) & \alpha_2 & \alpha_2\\
*(gray) & *(gray) & *(gray) & *(gray)
\end{ytableau}\,}.
\]
\end{example}

Note that if a hook-valued tableau \( T \) has \( a \) arms and
\( \ell \) legs, then \( \A^{a}(T) = \A^{a+1}(T) = \cdots \) and
\( \L^{\ell}(T) = \L^{\ell+1}(T) = \cdots \). Thus,
\( \U_{\L^N\A^M}(T) \) is the same for all \( N\geqslant \ell \) and
\( M\geqslant a \). We will write the result as
\( \U_{\L^\infty\A^\infty}(T) \). We define
\( \U_{\A^\infty\L^\infty}(T) \) similarly.

\subsection{Tableau switching}
\label{subsec.switching}

We recall results from a paper by Benkart, Sottile, and Stroomer~\cite{BSS} on tableau switching.

\begin{definition}\label{def:switching}
  Let \( T \) be an \( \alpha \)-column-strict and
  \( \beta \)-row-strict mixed tableau. Let \( u=T(i,j) \) be an entry
  of \( T \) and let \( v \) be the entry \( T(i,j+1) \) or
  \( T(i+1,j) \). Suppose \( u=\alpha_r \) and \( v=\beta_s \). Let
  \( T' \) be the mixed tableau obtained from \( T \) by interchanging
  \( u \) and \( v \). If \( T' \) is \( \alpha \)-column-strict and
  \( \beta \)-row-strict, such a process is called a \defn{switch}. If
  there is no possible switch, we say that \( T \) is \defn{fully
    switched}.
\end{definition}

Note that in a switch, we always move an \( \alpha \)-entry to the
north or east, and a \( \beta \)-entry to the south or west. Hence, if
we keep applying switches to a tableau, it eventually becomes fully
switched.

\begin{example}\label{exa:switching}
  In this example, we only write the indices, where the
  \( \alpha \)-indices are colored yellow. One possible switching procedure is as
  follows:
  \begin{align*}
  \ytableausetup{boxsize = 1em}
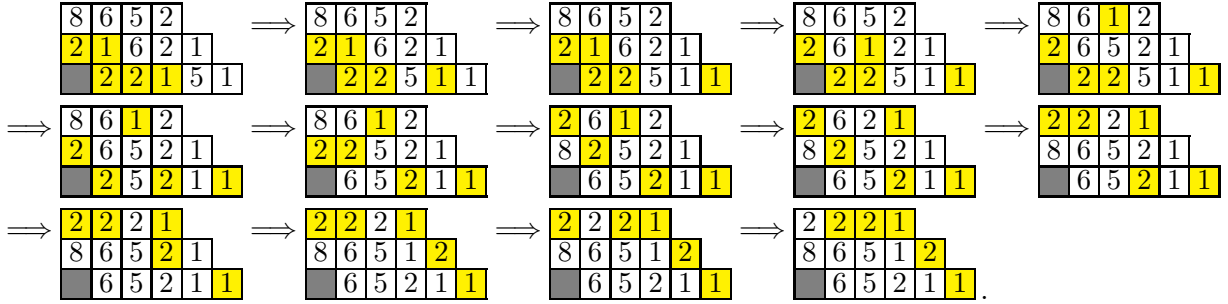
    \begin{ytableau}
       8 & 6 & 5 & 2 \\
      *(yellow)2 & *(yellow)1 & 6 & 2 & 1 \\
      *(gray) & *(yellow)2 & *(yellow)2 & *(yellow)1 & 5 & 1
    \end{ytableau}
    &\Longrightarrow    
    \begin{ytableau}
       8 & 6 & 5 & 2 \\
      *(yellow)2 & *(yellow)1 & 6 & 2 & 1 \\
      *(gray) & *(yellow)2 & *(yellow)2 & 5 & *(yellow)1 & 1
    \end{ytableau}
    \Longrightarrow    
    \begin{ytableau}
       8 & 6 & 5 & 2 \\
      *(yellow)2 & *(yellow)1 & 6 & 2 & 1 \\
      *(gray) & *(yellow)2 & *(yellow)2 & 5 & 1 & *(yellow)1
    \end{ytableau}
    \Longrightarrow    
    \begin{ytableau}
       8 & 6 & 5 & 2 \\
      *(yellow)2 & 6 & *(yellow)1 & 2 & 1 \\
      *(gray) & *(yellow)2 & *(yellow)2 & 5 & 1 & *(yellow)1
    \end{ytableau}
    \Longrightarrow    
    \begin{ytableau}
       8 & 6 & *(yellow)1 & 2 \\
      *(yellow)2 & 6 & 5 & 2 & 1 \\
      *(gray) & *(yellow)2 & *(yellow)2 & 5 & 1 & *(yellow)1
    \end{ytableau}\\
    \Longrightarrow    
    \begin{ytableau}
       8 & 6 & *(yellow)1 & 2 \\
      *(yellow)2 & 6 & 5 & 2 & 1 \\
      *(gray) & *(yellow)2 & 5 & *(yellow)2 & 1 & *(yellow)1
    \end{ytableau}
    &\Longrightarrow    
    \begin{ytableau}
       8 & 6 & *(yellow)1 & 2 \\
      *(yellow)2 & *(yellow)2 & 5 & 2 & 1 \\
      *(gray) & 6 & 5 & *(yellow)2 & 1 & *(yellow)1
    \end{ytableau}
    \Longrightarrow    
    \begin{ytableau}
       *(yellow)2 & 6 & *(yellow)1 & 2 \\
       8& *(yellow)2 & 5 & 2 & 1 \\
      *(gray) & 6 & 5 & *(yellow)2 & 1 & *(yellow)1
    \end{ytableau}
    \Longrightarrow    
    \begin{ytableau}
       *(yellow)2 & 6 & 2 & *(yellow)1 \\
       8& *(yellow)2 & 5 & 2 & 1 \\
      *(gray) & 6 & 5 & *(yellow)2 & 1 & *(yellow)1
    \end{ytableau}
    \Longrightarrow 
    \begin{ytableau}
       *(yellow)2 & *(yellow)2 & 2 & *(yellow)1 \\
       8& 6 & 5 & 2 & 1 \\
      *(gray) & 6 & 5 & *(yellow)2 & 1 & *(yellow)1
    \end{ytableau} \\
    \Longrightarrow    
    \begin{ytableau}
       *(yellow)2 & *(yellow)2 & 2 & *(yellow)1 \\
       8& 6 & 5 & *(yellow)2 & 1 \\
      *(gray) & 6 & 5 & 2 & 1 & *(yellow)1
    \end{ytableau} 
    &\Longrightarrow    
    \begin{ytableau}
       *(yellow)2 & *(yellow)2 & 2 & *(yellow)1 \\
       8& 6 & 5 & 1 & *(yellow)2 \\
      *(gray) & 6 & 5 & 2 & 1 & *(yellow)1
    \end{ytableau} 
    \Longrightarrow    
    \begin{ytableau}
       *(yellow)2 & 2 & *(yellow)2 & *(yellow)1 \\
       8& 6 & 5 & 1 & *(yellow)2 \\
      *(gray) & 6 & 5 & 2 & 1 & *(yellow)1
    \end{ytableau} 
    \Longrightarrow    
    \begin{ytableau}
       2 & *(yellow)2 & *(yellow)2 & *(yellow)1 \\
       8& 6 & 5 & 1 & *(yellow)2 \\
      *(gray) &  6 & 5 & 2 & 1 & *(yellow)1
    \end{ytableau}\, \raisebox{-.8cm}{.}
  \end{align*}
\end{example}

The following theorem will be used as an important ingredient
when we prove our main results.

\begin{theorem} \cite[Theorem~2.2]{BSS}
\label{thm.switch}
Let \( T \) be an \( \alpha \)-column-strict, \( \beta \)-row-strict,
and \( (\alpha,\beta) \)-sorted mixed tableau. We apply switches to
\( T \) until it is fully switched. Then, the resulting tableau is
\( \alpha \)-column-strict, \( \beta \)-row-strict, and
\( (\beta,\alpha) \)-sorted. Furthermore, it is independent of the sequence of
switches that produced it.
\end{theorem}

\begin{corollary}\label{cor:switch}
\mbox{}
\begin{enumerate}
\item Let \( T \) be an \( \alpha \)-column-strict,
  \( \beta \)-row-strict, and \( (\alpha,\beta) \)-sorted mixed
  tableau. Let \( X_1 \) and \( X_2 \) be fully switched tableaux obtained from
  \( T \) by some sequences of switches.
  Then \( X_1=X_2 \).
\item Let \( T_1 \) and \( T_2 \) be \( \alpha \)-column-strict,
  \( \beta \)-row-strict, and \( (\alpha,\beta) \)-sorted mixed
  tableaux. Suppose that \( X \) is a fully switched tableau obtained from both
  \( T_1 \) and \( T_2 \) by some
  sequences of switches. Then \( T_1=T_2 \).
\end{enumerate}

\end{corollary}

\begin{proof}
  The first statement is immediate from \Cref{thm.switch}. The second
  statement follows from the first statement if we consider the
  tableaux obtained by rotating \( T \), \( X_1 \), and \( X_2 \) by
  \( 180^\circ \) and replacing each index \( i \) by \( N+1-i \),
  where \( N \) is the largest index in \( T \).
\end{proof}

\begin{remark}
  We note that in \cite{BSS} different terminologies are used
  and the order is reversed:
  \begin{itemize}
  \item A perforated pair is a mixed tableau that is \( \alpha \)-column-strict and \( \beta \)-column-strict.
  \item A \( t \)-perforated pair is a mixed tableau that is \( \alpha \)-column-strict and \( \beta \)-row-strict.
  \end{itemize}
  We also note that, unlike \Cref{thm.switch}, in
  \cite[Theorem~2.2]{BSS}, it is assumed that \( T \) is
  \( \alpha \)-column-strict and \( \beta \)-column-strict, and remains
  so throughout the switch process. However, as mentioned in the last
  paragraph on page 22 of \cite{BSS}, it is still true if we replace the
  requirement with \( \alpha \)-column-strict and
  \( \beta \)-row-strict.
\end{remark}

\section{Main results}
\label{section.proof}

In this section, we state and prove our main results. In Section~\ref{subsec.armleg}, we study the effects of interchanging leg
and arm uncrowding on the recording tableau, which we relate to a particular sequence of tableau switching. 
In Section~\ref{subsec.ggjdt}, we define the Goulden-Greene jeu de taquin (GG-jdt) algorithm and show that it can be formulated 
as tableau switching discussed in Section~\ref{subsec.switching}. In
Section~\ref{subsec.image}, we finally state and prove our main
results regarding
the equivalence of combinatorial models for refined canonical stable Grothendieck polynomials
and the image of the uncrowding algorithm.

\subsection{Tableau switching on the recording tableau}
\label{subsec.armleg}
In this subsection, we study the effects of first performing leg-uncrowding and then arm-uncrowding on a hook-valued tableau in comparison to 
the order of first applying arm-uncrowding and then leg-uncrowding as studied in~\cite{PPPS}. To characterize the changes to the recording tableaux, 
we need to define a particular sequence of tableau switches on an $(\alpha, \beta)$-tableau.

\begin{definition} 
\label{def:shuffle} 
Let $Q$ be a flagged-mixed tableau that is \( \alpha \)-column-strict,
\( \beta \)-row-strict, and \( (\alpha,\beta) \)-sorted. The (jeu de
taquin) \defn{shuffle} of $Q$, denoted $\shuff(Q)$, is the tableau
obtained as follows:
\begin{enumerate}
\item Find the elements with weight $\alpha$ having an element of weight $\beta$ above or to the right.
\item Among them, choose the rightmost element that has the smallest index, say $\alpha_i$.
\item Continue the following `switching process' until the cells above and to the right of $\alpha_i$ have a weight $\alpha$ or are empty. If only one of the cells directly above or to the right of $\alpha_i$ has weight $\beta$, then switch $\alpha_i$ with this element. Otherwise, let $\beta_k$ and $\beta_j$ be the entries above and to the right of $\alpha_i$, respectively, and perform one of the following switches
 \[
\ytableausetup{boxsize = 2em}
    \begin{ytableau}
      \beta_k & \none \\
      \alpha_i & \beta_j
    \end{ytableau}
    \mapsto
    \begin{ytableau}
      \alpha_i & \none \\
      \beta_k & \beta_j
    \end{ytableau}
    \quad\mbox{if \( k > j \)},\qquad
    \begin{ytableau}
     \beta_k & \none \\
     \alpha_i & \beta_j
    \end{ytableau}
    \mapsto
    \begin{ytableau}
     \beta_k & \none \\
     \beta_j & \alpha_i
    \end{ytableau}
    \quad\mbox{if \( k \leqslant j \).}
  \] 
\item Repeat steps (1)-(3) until there is no cell having weight $\alpha$ with an element of weight $\beta$ directly above or to its right.
\end{enumerate}
\end{definition}

Note that \( \shuff(Q) \) can be seen to be equivalent to the shuffling procedure outlined in \cite{BSS} via~\cite[Lemma~2.7]{Haiman1992} 
and~\cite[Chapter 2]{Sh}. By~\cite[Proposition~2.5]{BSS},  \( \shuff(Q) \)  is a special case of tableau switching, hence \( \shuff(Q) \) is 
\( \alpha \)-column-strict, \( \beta \)-row-strict, and \( (\beta,\alpha) \)-sorted.

\begin{example}\label{exa:shuffle}
  The following shows the process of the shuffle, where boxes
  containing an \( \alpha \) are colored in yellow:
  \begin{align*}
  \ytableausetup{boxsize = 1.5em}
  Q= \begin{ytableau}
       \beta_8 & \beta_6 & \beta_5 & \beta_2 \\
      *(yellow)\alpha_2 & *(yellow)\alpha_1 & \beta_6 & \beta_2 & \beta_1 \\
      *(gray) & *(yellow)\alpha_2 & *(yellow)\alpha_2 & *(yellow)\alpha_1 & \beta_5 & \beta_1
    \end{ytableau}
    &\Longrightarrow    
    \begin{ytableau}
       \beta_8 & \beta_6 & \beta_5 & \beta_2 \\
      *(yellow)\alpha_2 & *(yellow)\alpha_1 & \beta_6 & \beta_2 & \beta_1 \\
      *(gray) & *(yellow)\alpha_2 & *(yellow)\alpha_2 & \beta_5 & \beta_1 & *(yellow)\alpha_1
    \end{ytableau}
    \Longrightarrow    
    \begin{ytableau}
       \beta_8 & \beta_6 & \beta_2 & *(yellow)\alpha_1 \\
      *(yellow)\alpha_2 & \beta_6  & \beta_5 & \beta_2 & \beta_1 \\
      *(gray) & *(yellow)\alpha_2 & *(yellow)\alpha_2 & \beta_5 & \beta_1 & *(yellow)\alpha_1
    \end{ytableau} \\
    \Longrightarrow    
    \begin{ytableau}
       \beta_8 & \beta_6 & \beta_2 & *(yellow)\alpha_1 \\
      *(yellow)\alpha_2 & \beta_6  & \beta_5 & \beta_1 & *(yellow)\alpha_2 \\
      *(gray) & *(yellow)\alpha_2 & \beta_5 & \beta_2 & \beta_1 & *(yellow)\alpha_1
    \end{ytableau}
    &\Longrightarrow    
    \begin{ytableau}
       \beta_8 & \beta_2 & *(yellow)\alpha_2 & *(yellow)\alpha_1 \\
      *(yellow)\alpha_2 & \beta_6 & \beta_5 & \beta_1 & *(yellow)\alpha_2 \\
      *(gray) &  \beta_6 & \beta_5 & \beta_2 & \beta_1 & *(yellow)\alpha_1
    \end{ytableau}
    \Longrightarrow    
    \begin{ytableau}
       \beta_2 & *(yellow)\alpha_2 & *(yellow)\alpha_2 & *(yellow)\alpha_1 \\
       \beta_8& \beta_6 & \beta_5 & \beta_1 & *(yellow)\alpha_2 \\
      *(gray) &  \beta_6 & \beta_5 & \beta_2 & \beta_1 & *(yellow)\alpha_1
    \end{ytableau}= \shuff(Q).
  \end{align*}
  Note this is an example of tableau switching,
  and the result \( \shuff(Q) \) is the same as the
  last tableau in \Cref{exa:switching}.
\end{example}

\begin{definition}\label{def:type}
  Let \( T \) be a hook-valued tableau. For a word
  \( w=f_k^{(\epsilon_k)} \cdots f_1^{(\epsilon_1)} \) with
  \( f_i\in \{\Abump,\Lbump\} \) and \( \epsilon_i\in \{0,1\} \), we
  say that \( T \) is \defn{of type} \( w \) if, for each
  \( i= 1,\dots,n \),
  \[
    |\shape(f_{i} \circ \cdots \circ f_1 (T))|
    = |\shape(f_{i-1} \circ \cdots \circ f_1 (T))| + \epsilon_i.
  \]
\end{definition}

\begin{example}\label{ex:type}
  The following shows the hook-valued tableaux \( T\), \( \Abump(T) \), and
  \( \Lbump(T) \):
\[{\def\mc#1{\makecell[lb]{#1}}
\,{\begin{array}[lb]{*{4}{|l}|}\cline{1-1}
\mc{55}\\\cline{1-1}
\mc{4}\\\cline{1-3}
\mc{222}&\mc{3\\2}&\mc{4\\3}\\\cline{1-4}
\mc{\phantom{2}\\1}&\mc{\phantom{2}\\1}&\mc{\phantom{3}\\2}&\mc{3\\2}\\\cline{1-4}
\end{array}\,},\quad
\,{\begin{array}[lb]{*{4}{|l}|}\cline{1-1}
\mc{5}\\\cline{1-2}
\mc{4}&\mc{5}\\\cline{1-3}
\mc{222}&\mc{3\\2}&\mc{4\\3}\\\cline{1-4}
\mc{\phantom{2}\\1}&\mc{\phantom{2}\\1}&\mc{\phantom{3}\\2}&\mc{3\\2}\\\cline{1-4}
\end{array}\,},\quad
\,{\begin{array}[lb]{*{4}{|l}|}\cline{1-1}
\mc{55}\\\cline{1-2}
\mc{4}&\mc{4}\\\cline{1-3}
\mc{\phantom{3}\\222}&\mc{3\\2}&\mc{\phantom{4}\\3}\\\cline{1-4}
\mc{\phantom{2}\\1}&\mc{\phantom{2}\\1}&\mc{\phantom{3}\\2}&\mc{3\\2}\\\cline{1-4}
\end{array}\,}.
}
\]
As \( \Abump(T) \) and \( \Lbump(T) \) created new cells, \(T\) has both type \(\Abump^{(1)}\) and \(\Lbump^{(1)}\). Note that \( \Abump(T) \) 
has type \(\Lbump^{(0)}\) and \( \Lbump(T) \) has type \(\Abump^{(1)}\). Thus, \(T\) also has type \(\Lbump^{(0)}\Abump^{(1)}\) and 
\(\Abump^{(1)}\Lbump^{(1)}\).
\end{example}

The following lemma shows that \( \Abump \) and \( \Lbump \)
essentially commute.

\begin{lemma}\label{lem:AL=LA}
  Let $T$ be a hook-valued tableau with at least one arm and at least one leg.
\begin{enumerate}
  \item If \( T \) has types \( \Lbump^{(1)}\Abump^{(1)} \) and \( \Abump^{(0)}\Lbump^{(1)} \), then
    it also has type \( \Abump^{(1)}\Abump^{(0)}\Lbump^{(1)} \), and
    \[
      \Lbump \circ \Abump(T) = \Abump^2 \circ \Lbump(T) .
    \]
  \item If \( T \) has types \( \Lbump^{(0)}\Abump^{(1)} \) and \( \Abump^{(1)}\Lbump^{(1)} \), then it
    also has type \( \Lbump^{(1)}\Lbump^{(0)}\Abump^{(1)} \) and
    \[
      \Lbump ^2 \circ \Abump(T) = \Abump \circ \Lbump(T).
    \]
  \item Suppose that \( T \) is not in the first two cases above. If
    \( T \) has type \( \Lbump^{(x)}\Abump^{(y)} \) for some \( x,y\in \{0,1\} \), then
    it also has type \( \Abump^{(y)}\Lbump^{(x)} \) and
    \[
      \Lbump\circ \Abump(T) = \Abump\circ \Lbump(T).
    \]
  \end{enumerate}
  \end{lemma}

\begin{proof}
Define $a$ and $\ell$ to be the arm and leg entries in $T$ bumped by $\Abump$ and $\Lbump$, respectively. Additionally, let $(i,j)$ be the cell containing 
the arm entry $a$ in $T$. We break into five cases based on where $a$ and $\ell$ are bumped by $\Abump$ and $\Lbump$.

\smallskip

\noindent
\textbf{Case 1:}
Assume that $a$ and $\ell$ are in the same cell, $a < \ell$, and $\Abump$ bumps $a$ into the cell $(i,j+1)$. In this case, $\ell$ is moved to cell $(i,j+1)$ by $\Abump$ where it remains the topmost leg that is rightmost in its row. As the bumping of $a$ by $\Abump$ does not affect row $i+1$ and the bumping of $\ell$ does not affect column $j+1$, we have $\Lbump \circ \Abump(T) = \Abump \circ \Lbump(T)$. A similar argument holds when $a$ and $\ell$ are in the same cell $(i,j)$, $a \geqslant \ell$, and $\Lbump$ bumps $\ell$ into the cell $(i+1,j)$.

In this case, \( T \) is of type \( \Lbump^{(y)}\Abump^{(x)} \) and of type \( \Abump^{(x)}\Lbump^{(y)} \), where \( x,y\in \{0,1\} \). 

\smallskip

\noindent
\textbf{Case 2:}
Assume that $a$ bumped $\ell$ in $\Abump(T)$. By the assumption, $\ell$ must reside in some cell $(i', j+1)$ in $T$ where $i' < i$. As $a$ bumped $\ell$, $a$ is now the topmost leg that is rightmost in its row within $\Abump(T)$, and $\ell$ is now an arm of the cell $(i', j+1)$.
Moreover, $a$ must be strictly smaller than the entry in cell $(i'+1, j+1)$, which is possibly empty, and weakly greater than all of the entries in cell $(i'+1,j)$.
Thus, $\Lbump$ acts on $\Abump(T)$ by leg bumping $a$ into cell $(i'+1, j+1)$ and lifting $\ell$ up to cell $(i'+1, j+1)$ as an arm. On the other hand, $\Lbump$ acts on $T$ by bumping $\ell$ into the cell $(i'+1, j+1)$, which is possibly empty, since $\ell$ is weakly greater than all of the entries in cell $(i'+1, j)$.
The arm $a$ is still the rightmost arm that is topmost in its column in $\Lbump(T)$. Thus, $a$ bumps $\ell$ in cell $(i'+1, j+1)$ in $\Abump \circ \Lbump(T)$. We see that $\Lbump \circ \Abump(T) = \Abump \circ \Lbump(T)$. A similar argument holds when $\ell$ bumps $a$ in $\Lbump(T)$.

In the first case, \( T \) is of type \( \Lbump^{(y)}\Abump^{(0)} \) and of type \( \Abump^{(0)}\Lbump^{(y)} \), where \( y\in \{0,1\} \).
In the second case, \( T \) is of type \( \Lbump^{(0)}\Abump^{(x)} \) and of type \( \Abump^{(x)}\Lbump^{(0)} \), where \( y\in \{0,1\} \).

\smallskip

\noindent
\textbf{Case 3:}
Assume that $a$ and $\ell$ bump the same nonempty entry, $x$, in $\Abump(T)$ and $\Lbump(T)$, respectively. This implies that in the tableau $T$, $\ell$ is contained in some cell $(i', j')$ where $i' \leqslant i-1$ and $j' \geqslant j+1$ and $x$ is contained in cell $(i'+1, j+1)$. If $\ell < a$, we see that in both $\Lbump \circ \Abump(T)$ and $\Abump \circ \Lbump(T)$, $a$ becomes the leg of cell $(i'+1, j+1)$, $\ell$ becomes the hook entry of $(i'+1, j+1)$, and $x$ becomes the arm of $(i'+1, j+1)$. If $a \leqslant \ell$, the entries $x$, $a$, and $\ell$ become the leg, hook entry, and arm, respectively, of the cell $(i'+1, j+1)$ in both $\Lbump \circ \Abump(T)$ and $\Abump \circ \Lbump(T)$. Thus, we have $\Lbump \circ \Abump(T) = \Abump \circ \Lbump(T)$.

In this case, \( T \) is of types \( \Lbump^{(0)}\Abump^{(0)} \) and \( \Abump^{(0)}\Lbump^{(0)} \).

\smallskip

\noindent
\textbf{Case 4:} Assume now that $a$ and $\ell$ are bumped into the
same empty cell by $\Abump$ and $\Lbump$, respectively. As in the
previous paragraph, $\ell$ is contained within some cell $(i', j')$ in
$T$ where $i' \leqslant i-1$ and $j' \geqslant j+1$ and the empty cell
must have coordinate $(i'+1, j+1)$.

If $a \leqslant \ell$, we have
$\Lbump \circ \Abump(T) = \Abump^2 \circ \Lbump(T)$ where $a$ is in
cell $(i'+1, j+1)$ and $\ell$ is in cell $(i'+1, j+2)$. In this case,
\( T \) is of types \( \Lbump^{(1)}\Abump^{(1)} \), \( \Abump^{(0)}\Lbump^{(1)} \), and \( \Abump^{(1)}\Abump^{(0)}\Lbump^{(1)} \).

However, if $\ell < a$, then
$\Lbump ^2 \circ \Abump(T) = \Abump \circ \Lbump(T)$ where $\ell$ is
in cell $(i'+1, j+1)$ and $a$ is in cell $(i'+2, j+1)$.
In this case,
\( T \) is of types \( \Abump^{(1)}\Lbump^{(1)} \), \( \Lbump^{(0)}\Abump^{(1)} \), and \( \Lbump^{(1)}\Lbump^{(0)}\Abump^{(1)} \).

\smallskip

\noindent
\textbf{Case 5:} In all other cases, the bumpings of $a$ and $\ell$ do
not interact, which implies that $\Abump$ and $\Lbump$ trivially
commute. In this case, if \( T \) has type
\( \Lbump^{(x)}\Abump^{(y)} \) for some \( x,y\in \{0,1\} \), then it
also has type \( \Abump^{(y)}\Lbump^{(x)} \) and
\( \Lbump\circ \Abump(T) = \Abump\circ \Lbump(T) \).

\smallskip

The above cases imply the desired result.
\end{proof}

We are now ready to discuss the changes to both the insertion and recording tableaux after swapping the order of the leg and arm-uncrowding 
algorithms to a hook-valued tableau.

\begin{proposition} \label{prop:commute} Let $T$ be a hook-valued
  tableau and set $(P_1, Q_1) = \U_{\L\A}(T)$ and
  $(P_2, Q_2) = \U_{\A\L}(T)$. Then the pairs of tableaux
  $(P_1, Q_1)$ and $(P_2, Q_2)$ satisfy the following relations:
\begin{enumerate}
\item $P_2 = P_1$ and 
\item $Q_2 = \shuff(Q_1)$.
\end{enumerate}
\end{proposition}

\begin{proof}
  For the first statement, since \( P_1 = \L\circ \A(T) \) and
  \( P_2 = \A\circ \L(T) \), we must show that \( \L\circ \A(T) = \A\circ \L(T) \). If
  \( T \) has no arm or no leg, then one of \( \A \) and \( \L \)
  becomes the identity map, hence \( \L\circ \A(T) = \A\circ \L(T) \). We now
  assume that \( T \) has at least one arm and at least one leg.

  Recall that if a hook-valued tableau \( H \) has at least one arm
  (resp.~leg), then \( \A(H)=\Abump^k(H) \)
  (resp.~\( \L(H)=\Lbump^k(H) \)), where \( k \) is the smallest
  integer such that \( \Abump^k(H) \) (resp.~\( \Lbump^k(H) \)) has
  one more cell than \( H \). Thus
  \( \L\circ \A(T) = \Lbump^{m}\circ \Abump^{n}(T) \) for the smallest positive
  integers \( m \) and \( n \) such that
  \begin{align*}
    |\shape(\Abump^{n}(T))| &> |\shape(T)|,\\
    |\shape(\Lbump^{m}\Abump^{n}(T))| &> |\shape(\Abump^{n}(T))|.
  \end{align*}
  Equivalently, \( T \) has type 
  \( \Lbump^{(1)} \overbrace{\Lbump^{(0)} \cdots \Lbump^{(0)}}^{m-1}
\Abump^{(1)} \overbrace{\Abump^{(0)} \cdots \Abump^{(0)}}^{n-1} \).

Our goal is to move all \( \Abump \)'s to the left of all
\( \Lbump \)'s in \( \Lbump^{m}\circ \Abump^{n}(T) \). We claim that
the leftmost \( \Abump \) can always be moved to the left until it
becomes the leftmost operator without changing the resulting tableau.
We consider three cases.

\smallskip \noindent
\textbf{Case 1:} \( m=1 \). Then \( \Abump^{n-1}(T) \) has type
\( \Lbump^{(1)}\Abump^{(1)} \). By \Cref{lem:AL=LA},
\[
  \Lbump^{m}\circ \Abump^{n}(T) = \Lbump\circ \Abump(\Abump^{n-1}(T))
  =
  \begin{cases}
    \Abump^2\circ \Lbump(\Abump^{n-1}(T))
    & \mbox{if \( \Abump^{n-1}(T) \) has type \( \Abump^{(0)}\Lbump^{(1)} \)},\\
    \Abump\circ \Lbump(\Abump^{n-1}(T))
    & \mbox{otherwise.}
  \end{cases}
\]
In either case, the claim holds.

\smallskip \noindent
\textbf{Case 2:} \( m=2 \). Then \( \Abump^{n-1}(T) \) has type
\( \Lbump^{(0)}\Abump^{(1)} \). By \Cref{lem:AL=LA},
\[
  \Lbump^{m}\circ \Abump^{n}(T) = \Lbump^2\circ \Abump(\Abump^{n-1}(T))
  =
  \begin{cases}
    \Abump\circ \Lbump(\Abump^{n-1}(T))
    & \mbox{if \( \Abump^{n-1}(T) \) has type \( \Abump^{(1)}\Lbump^{(1)} \)},\\
    \Lbump\circ \Abump\circ \Lbump(\Abump^{n-1}(T))
    & \mbox{otherwise.}
  \end{cases}
\]
The claim holds in the first case. In the second case, the claim
follows from the argument in Case~1.

\smallskip \noindent
\textbf{Case 3:} \( m\geqslant 3 \). Then \( \Abump^{n-1}(T) \) has type
\( \Lbump^{(0)}\Lbump^{(0)}\Abump^{(1)} \). Note that, by
\Cref{lem:AL=LA}, for a hook-valued tableau \( H \) of type
\( \Lbump^{(0)}\Abump^{(1)} \), we have
\( \Lbump\circ \Abump (H) = \Abump\circ \Lbump (H) \) unless \( H \)
has type \( \Abump^{(1)}\Lbump^{(1)} \), in which case \( H \) must
also have type \( \Lbump^{(1)}\Lbump^{(0)}\Abump^{(1)} \). Since
\( \Abump^{n-1}(T) \) has type
\( \Lbump^{(0)}\Lbump^{(0)}\Abump^{(1)} \), we have
\[
  \Lbump^{m}\circ \Abump^{n}(T)
  = \Lbump^{m-1}\circ \Lbump\circ \Abump(\Abump^{n-1}(T))
  = \Lbump^{m-1}\circ \Abump\circ \Lbump(\Abump^{n-1}(T)).
\]
Repeating the same argument with \( \Lbump^j\Abump^{n-1}(T) \) for
\( j=1,2,\dots,m-3 \), we obtain
\[
  \Lbump^{m}\circ \Abump^{n}(T)
= \Lbump^{2}\circ \Abump\circ \Lbump^{m-2}\circ \Abump^{n-1}(T).
\]
Then the claim follows from the same argument as in Case 2.

By the claim, we can write
\begin{equation}\label{eq:1}
  \Lbump^{m}\circ \Abump^{n}(T)
  = \Abump^{r}\circ \Lbump^s \circ \Abump^{n-1}(T)
\end{equation}
for some \( r\in \{1,2\} \) and \( s\in \{m-1,m\} \). Note that
\Cref{lem:AL=LA} implies that if a hook-valued tableau \( H \) has
type \( \Abump^{(0)} \), then we always have
\( \Lbump\circ \Abump (H) = \Abump\circ \Lbump (H) \). Since every
non-leftmost \( \Abump \) gives type \( \Abump^{(0)} \) in
\eqref{eq:1}, we have
\begin{equation}\label{eq:2}
  \Lbump^{m}\circ \Abump^{n}(T)
  = \Abump^{r}\circ \Lbump^s \circ \Abump^{n-1}(T)
  = \Abump^{r+n-1}\circ \Lbump^s(T).
\end{equation}
Moreover, by \Cref{lem:AL=LA}, the operator giving type
\( \Lbump^{(1)} \) (resp.~\( \Abump^{(1)} \)) in \eqref{eq:2} is
always the leftmost \( \Lbump \) (resp.~\( \Abump \)). Hence, we
obtain that \( T \) has type
\( \Abump^{(1)} \overbrace{\Abump^{(0)} \cdots \Abump^{(0)}}^{r+n-2}
\Lbump^{(1)} \overbrace{\Lbump^{(0)} \cdots \Lbump^{(0)}}^{s-1} \).
This implies that
  \begin{align*}
    \L\circ \A(T) = \Lbump^{m}\circ \Abump^{n}(T) = \Abump^{r+n-1}\circ \Lbump^{s}(T) = \A\circ \L(T),
  \end{align*}
 which is the first statement.

\medskip

We now prove the second statement.
Note that $Q_1$ and $Q_2$ are both comprised of two cells: one having weight $\alpha$ and the other having weight $\beta$. As $\A$ does not change the row index of any leg and $\L$ does not change the column index of any arm, the indices of $\alpha$ and $\beta$ in $Q_1$ are the same as the indices of $\alpha$ and $\beta$ in $Q_2$, respectively. It remains to consider the position of the cells having weight $\alpha$ and $\beta$ in $Q_1$ and $Q_2$.

Since $Q_1$ is comprised only of a cell with weight $\alpha$ and a cell with weight $\beta$, $\shuff(Q_1)$ switches the filling of the two cells if and 
only if the two cells are adjacent and otherwise leaves $Q_1$ unchanged. As $P_2 = P_1$ by \Cref{prop:commute}~(1), it follows that 
$Q_2 = \shuff(Q_1)$ if the two cells in $Q_1$ are adjacent.

Note that if \( \A \) (resp.~\( \L \)) starts with or bumps a cell
\( (i,j) \) during the process and creates a new cell \( (i',j') \) at
the end, then \( i'\leqslant i \) and \( j'>j \) (resp.~\( i'>i \) and
\( j'\leqslant j \)). Therefore, in the proof of
\Cref{prop:commute}~(1), it is straightforward to check that the two
cells in $Q_1$ are adjacent if and only if the new cells created in
$\A(T)$ and $\L(T)$ are in the same position (\textbf{Case 4} in the
proof of \Cref{lem:AL=LA}). This implies that if the two cells in
$Q_1$ are not adjacent, then $Q_2 = Q_1 = \shuff(Q_1)$. In either
case, we have the desired equality $Q_2 = \shuff(Q_1)$.
\end{proof}

\begin{theorem} \label{thm.shuff.q}
  Let $T$ be a hook-valued tableau and set $(P_1, Q_1) = \U_{\L^\infty\A^\infty} (T)$ and $(P_2, Q_2) = \U_{\A^\infty\L^\infty} (T)$. 
  Then the pairs of tableaux $(P_1, Q_1)$ and $(P_2, Q_2)$ satisfy the following relations:
  \begin{enumerate}
\item $P_2 = P_1$ and 
\item $Q_2 = \shuff(Q_1)$.
\end{enumerate}
\end{theorem}

\begin{proof}
Let $m$ and $n$ denote the arm and leg excesses of $T$, respectively. Note that $(P_1, Q_1) = \U_{\L^n\A^m}(T)$ and $(P_2, Q_2) = \U_{\A^m\L^n}(T)$. 
The fact that $P_2 = P_1$ follows directly from Proposition~\ref{prop:commute}~(1) by iteratively commuting the 
$\A$ operators through all of the $\L$'s. Thus, it remains to show $Q_2 = \shuff(Q_1)$.

During the `switching process' in the jeu de taquin shuffle, an element having weight $\alpha$ is switched with $\beta$-weighted elements having 
successively smaller indices. Thus, commuting an $\A$ operator through all of $\L$'s corresponds to performing Step (3) in the jeu de taquin shuffle (in \Cref{def:shuffle}) 
by Proposition ~\ref{prop:commute}~(2). Moreover, the order of the $\alpha$'s considered during the jeu de taquin shuffle is exactly opposite of the
 order in which they are recorded when applying $\A^{\infty}$. Therefore, iteratively commuting the $\A$ operators through all of the $\L$'s exactly 
 corresponds to performing the jeu de taquin shuffle on $Q_1$ giving the desired equality, $Q_2 = \shuff(Q_1)$.
\end{proof}

\subsection{Goulden-Greene jeu de taquin}
\label{subsec.ggjdt}

In this subsection, we define the Goulden--Greene jeu de taquin
algorithm~\cite{GG} (a process they refer to as the modified jeu de taquin) using our notation.
This map was also studied by Krattenthaler \cite{Kra}.
We then show that it can be realized as a tableau switching.

\begin{definition}[GG-jdt slides]\label{def:GG-jdt-slides}
  Let \( T \) be an \( \alpha \)-column-strict and
  \( \beta \)-row-strict mixed tableau. We say that the entry
  \( T(i,j) \) is \defn{out of order} if at least one of the following
  conditions holds:
  \begin{enumerate}
  \item \( T(i,j) = \alpha_r \) and \( T(i,j+1) = \beta_s \) for some \( r \) and \( s \) with \( r<s + (j+1) - i \).
  \item \( T(i,j) = \alpha_r \) and \( T(i+1,j) = \beta_t \) for some \( r \) and \( t \) with \( r\leqslant t + j - (i+1) \).
  \end{enumerate}
  If Condition (1) (resp. Condition (2)) holds, we define the
  \defn{horizontal slide} (resp. \defn{vertical slide}) at \( (i,j) \)
  to be the operation that swaps \( \alpha_r \) and \( \beta_s \). 

  If \( T(i,j) \) is out of order, the \defn{GG-jdt slide} at
  \( (i,j) \) is the unique available operation between the horizontal
  slide and the vertical slide at \( (i,j) \) such that the resulting
  tableau is still \( \beta \)-row-strict. More precisely, if only
  Condition (1) (resp.~(2)) holds, then the GG-jdt slide at
  \( (i,j) \) is the horizontal (resp.~vertical) slide at \( (i,j) \).
  If both Conditions (1) and (2) hold, then \( T(i,j) = \alpha_r \),
  \( T(i,j+1) = \beta_s \), and \( T(i+1,j) = \beta_t \) for some
  \( r \), \( s \), and \( t \) with \( r<s + (j+1) - i \) and
  \( r\leqslant t + j - (i+1) \). In this case, the GG-jdt slide at
  \( (i,j) \) is the horizontal slide at \( (i,j) \) if \( t\leqslant s \)
  and the vertical slide at \( (i,j) \) if \( t > s \):
  \ytableausetup{boxsize = 2em}
  \[
    \begin{ytableau}
      \beta_t \\
      \alpha_r & \beta_s
    \end{ytableau}
    \mapsto
    \begin{ytableau}
      \beta_t \\
      \beta_s & \alpha_r
    \end{ytableau}
    \quad\mbox{if \( t \leqslant s \)},\qquad
    \begin{ytableau}
      \beta_t \\
      \alpha_r & \beta_s
    \end{ytableau}
    \mapsto
    \begin{ytableau}
      \alpha_r \\
      \beta_t & \beta_s
    \end{ytableau}
    \quad\mbox{if \( t > s \)}.
  \]
\end{definition}

\begin{definition}[GG-jdt map]\label{def:GG-jdt}
  Let \( T \) be a mixed tableau that is \( \alpha \)-column-strict,
  \( \beta \)-row-strict, and \( (\alpha,\beta) \)-sorted. The
  \defn{GG-jdt map} is the map \( \jdt \) sending \( T \) to the
  tableau \( \jdt(T) \) obtained as follows:
\begin{enumerate}
\item Find the smallest \( r \) such that \( \alpha_r \) is out of
  order in \( T \). Let \( (i,j) \) be the rightmost cell containing
  such an \( \alpha_r \).
\item Apply the GG-jdt slide to \( T \) at \( (i,j) \).
\item Repeat (1)-(2) until no entries are out of order.
\end{enumerate}
\end{definition}

Observe that the GG-jdt slide is exactly a switch defined in
\Cref{def:switching} except that it does not require the resulting
tableau to be \( \alpha \)-column-strict and
\( \beta \)-column-strict. However, we will see that this condition
holds automatically. Hence, the GG-jdt map is a sequence of switches.

\begin{example}\label{exa:GGjdt}
  The following is a tableau \( Q \), where the row indices of each
  tableau are written on its left: \ytableausetup{boxsize = 1.5em}
  \[
    \begin{ytableau}
    \none[9]&  *(gray) &*(gray) &  \beta_8 & \beta_6 & \beta_5 & \beta_2 \\
    \none[8]&  *(gray) &*(gray) & \alpha_2 & \alpha_1 & \beta_6 & \beta_2 & \beta_1 \\
    \none[7]&  *(gray) &*(gray) & *(gray) & \alpha_2 & \alpha_2 & \alpha_1 & \beta_5 & \beta_1\\
    \none[6]&  *(gray) & *(gray) & *(gray) & *(gray) & *(gray) & *(gray) & *(gray) & *(gray) \\ 
    \none[\vdots]&  \none[\vdots] & \none[\vdots] & \none[\vdots] & \none[\vdots] & \none[\vdots] & \none[\vdots] & \none[\vdots]& \none[\vdots]\\ 
    \none[1]&  *(gray) & *(gray) & *(gray) & *(gray) & *(gray) & *(gray) & *(gray) & *(gray) 
    \end{ytableau} \, \raisebox{-3cm}{.}
  \]
  Note that \( Q \) is a flagged-mixed tableau that is
  \( \alpha \)-column-strict, \( \beta \)-row-strict, and
  \( (\alpha,\beta) \)-sorted. The GG-jdt map applied to \( Q \)
  proceeds as follows, where we truncate the first two columns and the
  first six rows of the tableau, and color the boxes containing an
  \( \alpha \) in yellow: \ytableausetup{boxsize = 1.5em}
  \begin{align*}
   Q = \begin{ytableau}
      \beta_8 & \beta_6 & \beta_5 & \beta_2 \\
      *(yellow)\alpha_2 & *(yellow)\alpha_1 & \beta_6 & \beta_2 & \beta_1 \\
      *(gray) & *(yellow)\alpha_2 & *(yellow)\alpha_2 & *(yellow)\alpha_1 & \beta_5 & \beta_1
    \end{ytableau}
    &\Longrightarrow    
    \begin{ytableau}
      \beta_8 & \beta_6 & \beta_5 & \beta_2 \\
      *(yellow)\alpha_2 & *(yellow)\alpha_1 & \beta_6 & \beta_2 & \beta_1 \\
      *(gray) & *(yellow)\alpha_2 & *(yellow)\alpha_2 & \beta_5 & \beta_1 & *(yellow)\alpha_1
    \end{ytableau}
    \Longrightarrow    
    \begin{ytableau}
      \beta_8 & \beta_6 & *(yellow)\alpha_1 & \beta_2 \\
      *(yellow)\alpha_2 & \beta_6 & \beta_5 & \beta_2 & \beta_1 \\
      *(gray) & *(yellow)\alpha_2 & *(yellow)\alpha_2 & \beta_5 & \beta_1 & *(yellow)\alpha_1
    \end{ytableau}\, \\
    \Longrightarrow    
    \begin{ytableau}
      \beta_8 & \beta_6 & *(yellow)\alpha_1 & \beta_2 \\
      *(yellow)\alpha_2 & \beta_6 & \beta_5 & \beta_2 & \beta_1 \\
      *(gray) & *(yellow)\alpha_2 & \beta_5 &  *(yellow)\alpha_2 & \beta_1 & *(yellow)\alpha_1
    \end{ytableau}
    &\Longrightarrow    
    \begin{ytableau}
      \beta_8 & \beta_6 & *(yellow)\alpha_1 & \beta_2 \\
      *(yellow)\alpha_2 & *(yellow)\alpha_2 & \beta_5 & \beta_2 & \beta_1 \\
      *(gray) & \beta_6 & \beta_5 &  *(yellow)\alpha_2 & \beta_1 & *(yellow)\alpha_1
    \end{ytableau} 
    \Longrightarrow    
    \begin{ytableau}
      *(yellow)\alpha_2 & \beta_6 & *(yellow)\alpha_1 & \beta_2 \\
      \beta_8 & *(yellow)\alpha_2 & \beta_5 & \beta_2 & \beta_1 \\
      *(gray) & \beta_6 & \beta_5 &  *(yellow)\alpha_2 & \beta_1 & *(yellow)\alpha_1
    \end{ytableau} =\jdt(Q).
  \end{align*}
  The tableau resulting from applying \( c_\beta^+ \) to \( \jdt(Q) \) is 
  \[
    c_{\beta}^+(\jdt(Q)) =
    \begin{ytableau}
      *(yellow)\alpha_2 & \beta_1 & *(yellow)\alpha_1 & \beta_{-1} \\
      \beta_3 & *(yellow)\alpha_2 & \beta_2 & \beta_0 & \beta_0 \\
      *(gray) & \beta_3 & \beta_3 &  *(yellow)\alpha_2 & \beta_1 & *(yellow)\alpha_1
    \end{ytableau},
  \]
  which is totally column-strict.

  We can compare this with the tableau switching in
  \Cref{exa:shuffle}. Note that this process is a sequence of
  switches. By the uniqueness of the fully switched tableau, see
  \Cref{cor:switch} (1), if we keep applying switches to
  \( \jdt(Q) \), then we get the same resulting tableau
  \( \shuff(Q) \) at the end of \Cref{exa:shuffle}.
\end{example}

The following properties of the GG-jdt map were stated in \cite[Section~3]{GG}
without proof. Krattenthaler \cite[Lemma~1]{Kra} gave complete proofs
of these results.

\begin{proposition}\label{pro:GGjdt}\
  \begin{enumerate}
  \item   The GG-jdt map is a bijection from the set of
  \( \alpha \)-column-strict, \( \beta \)-row-strict, and
  \( (\alpha,\beta) \)-sorted tableaux to the set of tableaux \( E \)
  such that \( c_\beta^+(E) \) is a totally column-strict tableau.
\item Each iteration of the GG-jdt slide results in an
  \( \alpha \)-column-strict and \( \beta \)-row-strict tableau.
  \end{enumerate}
\end{proposition}

By \Cref{pro:GGjdt}~(2), the GG-jdt map can be formulated as a
tableau switching procedure as defined in
Section~\ref{subsec.switching}.

\begin{proposition}\label{prop.perforated}
  Let \( T \) be a flagged-mixed tableau that is
  \( \alpha \)-column-strict, \( \beta \)-row-strict, and
  \( (\alpha,\beta) \)-sorted. Then
  \( \jdt(T) \) is obtained from
  \( T \) by a sequence of switches.
\end{proposition}

\subsection{Bijections and the image of the uncrowding algorithm}
\label{subsec.image}

In this subsection, we find a bijection that relates hook-valued
tableaux, exquisite tableaux, and a new class of tableaux called
biflagged tableaux.
We also show that GG-jdt is a bijection between
biflagged tableaux and exquisite tableaux.
As a corollary, we characterize the image of the
uncrowding algorithm \( \U_{\L^\infty\A^\infty} \) defined on
hook-valued tableaux.

\begin{definition}\label{def:BFT}
  A \defn{biflagged tableau} is a mixed tableau \( T \)
  satisfying the following conditions:
  \begin{enumerate}
  \item \( T \) is \( \alpha \)-column-strict, \( \beta \)-row-strict, and \( (\alpha,\beta) \)-sorted.
  \item Both \( T \) and \( \shuff(T) \) are flagged-mixed tableaux.
  \end{enumerate}
We denote by \( \BFT(\ml) \)
the set of biflagged tableaux of shape \( \ml \).
\end{definition}

\begin{example}\label{eg.biflagged}
When $\lambda = (2,1)$ and $\mu = (3,3,1)$, we have
\ytableausetup{boxsize = 1.4em}
\[
\BFT(\ml) = \left\{
\raisebox{1em}{\begin{ytableau}
\beta_2 \\
*(gray) & \alpha_1 & \beta_1 \\
*(gray) & *(gray) & \alpha_2
\end{ytableau}}\,,
\raisebox{1em}{\begin{ytableau}
\beta_1 \\
*(gray) & \alpha_1 & \beta_1 \\
*(gray) & *(gray) & \alpha_2
\end{ytableau}}\,,
\raisebox{1em}{\begin{ytableau}
\beta_2 \\
*(gray) & \alpha_1 & \alpha_1 \\
*(gray) & *(gray) & \alpha_2
\end{ytableau}}\,,
\raisebox{1em}{\begin{ytableau}
\beta_1 \\
*(gray) & \alpha_1 & \alpha_1 \\
*(gray) & *(gray) & \alpha_2
\end{ytableau}}
\right\}\,.
\]
Note that
\[
Q = \raisebox{1em}{\begin{ytableau}
\beta_2 \\
*(gray) & \alpha_1 & \beta_1 \\
*(gray) & *(gray) & \alpha_1
\end{ytableau}} \notin \BFT(\ml) \quad \text{ since } \quad \shuff(Q) =
\raisebox{1em}{\begin{ytableau}
\beta_2 \\
*(gray) & \alpha_1 & \alpha_1 \\
*(gray) & *(gray) & *(orange)\beta_1
\end{ytableau}} \notin \FMT,
\]
where the \( \beta_1 \) in the cell \( (1,3) \) does not satisfy the flagged condition.
\end{example}

We will show that GG-jdt is a bijection
from the set of biflagged tableaux \( \BFT(\ml) \) 
to the set of exquisite tableaux $\EXQ(\ml)$ defined in
\Cref{def:exq}. We denote by $\SSYT(\mu)$ the set of semistandard
Young tableaux of shape $\mu$. As usual, the weight \( \wt(P) \) of a
semistandard Young tableau \( P \) is the product of \( x_i \) for all
entries \( i \) in \( P \).

\begin{lemma}
\label{lemma.PQ}
  Let \(T\in \HVT(\lambda)\) and \( \U_{\L^\infty\A^\infty} (T) = (P,Q) \).
  Then \(P \in \SSYT(\mu)\) and \(Q\in \BFT(\ml)\) for a partition \(\mu\) with \(\lambda \subseteq \mu\). Furthermore, $\wt(T) = \wt(P)\wt(Q)$.
\end{lemma}

\begin{proof}
  By Definition~\ref{defn.uncrowd}, \( Q \) is
  \( \alpha \)-column-strict, \( \beta \)-row-strict, and
  \( (\alpha,\beta) \)-sorted, and \( Q\in\FMT \).
  Theorem~\ref{thm.shuff.q} states that
  \(\U_{\A^\infty\L^\infty} (T) = (P,\shuff(Q))\). Thus, by
  Definition~\ref{defn.uncrowd} again, \( \shuff(Q)\in\FMT \).
  Therefore $Q\in \BFT(\ml)$. The statement on the weight follows from
  Definition~\ref{defn.uncrowd}.
\end{proof}

\begin{lemma}\label{lem:BFT-EXQ}
Let \( Q\in \BFT(\ml) \). Then \( \jdt(Q)\in \EXQ(\ml) \).
\end{lemma}

\begin{proof}
  Let \(T = \jdt (Q)\). By \Cref{pro:GGjdt}~(1), \( c_{\beta}^+(T) \) is a
  totally column-strict tableau. Thus, by Definition~\ref{def:exq}, it
  remains to show that \( T \) is a flagged-mixed tableau.

  First, we need to show that if \( T(i,j) = \alpha_k \), then
  \( 0 < k < j \). Since \( T = \jdt (Q) \) and \( \jdt \) moves
  the \( \alpha \) entries weakly northeast, the entry
  \( \alpha_k \) at \( (i,j) \) in \( T \) came from the entry
  \( \alpha_k \) at \( (i',j') \) of \( Q \) for some \( i'\leqslant i \)
  and \( j'\leqslant j \). By the assumption \( Q\in \BFT(\ml) \), \( Q \)
  is a flagged-mixed tableau, which implies \( 0<k< j' \). Therefore,
  \( 0<k<j \) as desired.

  Now we need to show that if \( T(i,j) = \beta_r \), then
  \( 0< r < i \). By \Cref{prop.perforated}, \( T \) is obtained from
  \( Q \) by a sequence of switches. Let \( Q' \) be the fully
  switched tableau obtained from \( T \) by further applying as many
  switches as possible. Recall that switches move the \( \beta \)
  entries weakly southwest. Thus, the entry \( \beta_r \) at
  \( (i,j) \) in \( T \) will be moved to the entry \( \beta_r \) at
  \( (i',j') \) in \( Q' \) for some \( i'\leqslant i \) and \( j'\leqslant j \).
  Since both \( Q' \) and \( \shuff(Q) \) are fully switched tableaux
  starting from \( Q \), by Corollary~\ref{cor:switch}~(1), we obtain
  \( Q'=\shuff(Q) \). Since \(Q\in \BFT(\ml)\), \(\shuff(Q)\) is a
  flagged-mixed tableau, ensuring that the entry \( \beta_r \) at
  \( (i',j') \) in \( Q'=\shuff(Q) \) satisfies \( 0<r<i' \). Thus,
  \( 0<r<i \) as claimed. This completes the proof.
\end{proof}

By \Cref{lemma.PQ,lem:BFT-EXQ}, we can define the following map.

\begin{definition}\label{def:Phi}
We define the map
 \[
\Phi\colon \HVT(\lambda) \to \bigcup_{\mu \supseteq \lambda} \left( \SSYT(\mu)\times \EXQ(\ml) \right)
\]
as the composition of the following two maps:
\begin{equation}\label{eq:bijections}
  \HVT(\lambda)
  \xrightarrow{\U_{\L^\infty\A^\infty}}
 \displaystyle  \bigcup_{\mu \supseteq \lambda} \left( \SSYT(\mu)\times \BFT(\ml) \right)
  \xrightarrow{\mathrm{id}\times\jdt}
 \displaystyle \bigcup_{\mu \supseteq \lambda} \left( \SSYT(\mu)\times \EXQ(\ml) \right).
\end{equation}
In other words, \( \Phi(T) = (P,E) \), where
\( \U_{\L^\infty\A^\infty} (T) = (P,Q) \) and
\( E = \jdt(Q) \).
\end{definition}

We will show that the two maps \( \jdt \) and \( \Phi \) are
bijections. We first prove the injectivity of \( \jdt \).

\begin{lemma}\label{lem:Psi}
  The map \( \jdt \colon \BFT(\ml) \to \EXQ(\ml) \) is an injection.
\end{lemma}

\begin{proof}
  Suppose $Q, Q'\in \BFT(\ml)$ and \(E = \jdt(Q) = \jdt(Q') \). By
  Proposition~\ref{prop.perforated}, the map $\jdt$ is a partial
  tableau switching. Let \( X \) be the (unique) fully switched
  tableau obtained by continuing the tableau switch process on
  \( E \). This means that \( X \) is obtained from \( Q \) and also
  from \( Q' \) by tableau switching. By \Cref{cor:switch} (2), we
  obtain \( Q=Q' \). Therefore, the map \( \jdt \) is injective.
\end{proof}

\begin{theorem}\label{thm:2bij}
  The maps
  \[
    \jdt\colon \BFT(\ml) \to \EXQ(\ml)
  \]
  and 
 \[
\Phi\colon \HVT(\lambda) \to \bigcup_{\mu \supseteq \lambda} \left( \SSYT(\mu)\times \EXQ(\ml) \right)
\]
are bijections. Furthermore, both maps are weight-preserving, i.e., 
if \( \jdt(Q)=E \), then \( \wt(Q)=\wt(E) \), and 
if \( \Phi(T) = (T_1,T_2) \), then \( \wt(T) = \wt(T_1)\wt(T_2) \).
\end{theorem}

\begin{proof}
  By construction, both maps are weight-preserving. By \Cref{lem:Psi},
  \( \jdt \) is injective. In \cite[Corollary~3.32]{PPPS}, it is shown
  that \( \U_{\A^\infty} \) is injective. Since \( \U_{\L^\infty} \)
  is also injective, so is
  \( \U_{\L^\infty\A^\infty} = \U_{\L^\infty}\circ \U_{\A^\infty} \).
  Therefore,
  \( \Phi = (\mathrm{id}\times \jdt)\circ \U_{\L^\infty\A^\infty} \)
  is also injective. Observe that, by the construction of \( \Phi \),
  if \( \Phi(T) = (P,E) \), then \( \wt(T) = \wt(P)\wt(E) \). This
  implies that
  \begin{equation}\label{eq:3}
    \sum_{T\in\HVT(\lambda)} \wt(T)
    \preccurlyeq \sum_{\mu \supseteq \lambda}
    \sum_{P\in \SSYT(\mu)}  \wt(P)
    \sum_{E\in  \EXQ(\ml)} \wt(E),
  \end{equation}
  where \( A \preccurlyeq B \) means that \( B-A \) is a formal power
  series with nonnegative coefficients. By \Cref{thm:HVT} and
  \Cref{thm.exq}, both sides of \eqref{eq:3} are equal. This implies
  that \( \Phi \) is a bijection. Then, considering the maps in
  \eqref{eq:bijections}, we see that \( \jdt \) is a surjection, hence
  a bijection.
\end{proof}

\begin{example}\label{eg.bij}
The map \( \jdt \) sends the biflagged tableaux in Example~\ref{eg.biflagged} to
  the exquisite tableaux in \Cref{exa:1} in that order.
\end{example}

As a corollary, we now have a characterization of the image of the
uncrowding algorithm.

\begin{corollary}
\label{corollary.image}
  The uncrowding map
  \[
    \U_{\L^\infty\A^\infty}\colon \HVT(\lambda) \to \bigcup_{\mu \supseteq \lambda} \left( \SSYT(\mu)\times \BFT(\ml) \right)
  \]
  is a bijection. Furthermore, it is weight-preserving, i.e., if
  \( \U_{\L^\infty\A^\infty}(T) = (T_1,T_2) \), then
  \( \wt(T) = \wt(T_1)\wt(T_2) \).
\end{corollary}

\begin{proof}
  By \Cref{lemma.PQ}, \( \U_{\L^\infty\A^\infty} \) is weight-preserving.
  By \Cref{thm:2bij}, the map \( \Phi \), which is the composition of
  the two maps in \eqref{eq:bijections}, is a bijection. By
  \Cref{thm:2bij}, \( \mathrm{id}\times \jdt \) is a bijection, hence
  the other map \( \U_{\L^\infty\A^\infty} \) is also a bijection.
\end{proof}

\begin{example}
In \cite[Remark~3.28]{PPPS}, it is shown that
\[
\ytableausetup{notabloids,boxsize=1.6em}
(P_1,Q_1) = 
\left(
\raisebox{0.2em}{
\begin{ytableau}
  3 \\
  \substack{2\\1} & \substack{3\\2}
\end{ytableau}\;,
\begin{ytableau}
  *(gray) \\
  *(gray) & \alpha_1
\end{ytableau}\,}\right) \text{ is not in the image of } \U_{\A^\infty}.
\]
In other words, this pair of tableaux cannot be obtained by applying
arm uncrowding on any hook-valued tableau. We can check this using
\Cref{corollary.image}. To do so, suppose that
\( \U_{\A^\infty}(T)=(P_1,Q_1) \) for some hook-valued tableau
\( T \). Continuing the leg uncrowding process on \( P_1 \), we have
\[
  \begin{ytableau}
  3 \\
  \substack{2\\1} & \substack{3\\2}
\end{ytableau}
\xrightarrow{\L}
\begin{ytableau}
  3 & 3\\ 
  \substack{2\\1} & 2
\end{ytableau}
\xrightarrow{\L}
\raisebox{1.6em}{\begin{ytableau}
2\\
2 & 3\\
1 & 2
\end{ytableau}}\;,
\quad \mbox{hence} \quad
\U_{\L^\infty\A^\infty}(T)=(P,Q) =
\left(
\raisebox{1.0em}{\begin{ytableau}
2\\
2 & 3\\
1 & 2
\end{ytableau}}\;,
\raisebox{1.0em}{
  \begin{ytableau}
    \beta_1 \\
    *(gray) & \beta_1\\
    *(gray) & \alpha_1
  \end{ytableau}}
\right).
\]
However, \( Q \) is not a biflagged tableau since
\[
\ytableausetup{notabloids,boxsize=1.6em}
\shuff\left(\,
\raisebox{1em}{\begin{ytableau}
    \beta_1 \\
    *(gray) & \beta_1\\
    *(gray) & \alpha_1
  \end{ytableau}\,}
\right) = 
\raisebox{1em}{\begin{ytableau}
\beta_1\\
*(gray) & \alpha_1 \\
*(gray) & *(orange)\beta_1
\end{ytableau}} \notin \FMT.
\]
This is a contradiction to \Cref{corollary.image}. Therefore, the pair
\( (P_1,Q_1) \) of tableaux is not in the image of $\U_{\A^\infty}$,
which is consistent with the observation in~\cite{PPPS}.
\end{example}

Combining \Cref{thm:HVT}, \Cref{thm.exq}, and \Cref{corollary.image},
we have three different combinatorial descriptions of the refined
canonical stable Grothendieck polynomials.

\begin{corollary}
\label{corollary.grothendieck}
We have
\begin{align*}
G_\lambda(\bm{x}, \bm{\alpha},\bm{\beta}) &= \sum_{H\in \HVT(\lambda)}\wt(H)\\
&= \sum_{\mu \supseteq \lambda}s_\mu(\bm{x}) \sum_{E\in \EXQ(\ml)} \wt(E) \\
&= \sum_{\mu \supseteq \lambda}s_\mu(\bm{x}) \sum_{Q\in \BFT(\ml)} \wt(Q).
\end{align*}
\end{corollary}

\section{Further Study}
\label{sec:further-study}

In this section, we propose some open problems.
Recall that in \Cref{subsec.image}, we have proved that the map \( \Phi \) defined in
\Cref{def:Phi} is a bijection using \Cref{thm:HVT} and \Cref{thm.exq}.
It would be interesting to find a combinatorial proof without using
these theorems.

\begin{problem}
  Find a combinatorial proof of the bijectivity of the map \( \Phi \).
\end{problem}

Let \( G_\mu(\bm{x},\bm{\beta}) \) be the symmetric Grothendieck polynomial
\( G_\mu(\bm{x},\bm{\alpha},\bm{\beta}) \) with substitution
\( \alpha_i=0 \) for all \( i \). Iwao et
al.~\cite[Cor~4.16]{Iwao2024} showed that\footnote{Their Grothendieck
  polynomial is \( G_\lambda(\bm{x}, \bm{\alpha},\bm{\beta}) \) with
  \( \alpha_i \) replaced by \( -\alpha_i \). There are some sign
  errors in \cite{Iwao2024}.} if
\begin{equation}\label{eq:4}
  G_\lambda(\bm{x}, \bm{\alpha},\bm{\beta})
  = \sum_{\mu\supseteq \lambda} B_{\lambda}^\mu G_\mu(\bm{x},\bm{\beta}),
\end{equation}
then \( B_{\lambda}^\mu\in \ZZ_{\ge0}[\bm{\alpha},-\bm{\beta}] \).
There is also a combinatorial description for \( B_{\lambda}^\mu \);
see \cite[Remark 4.17]{Iwao2024}. It would be interesting to prove
this using variations of uncrowding algorithms. Note that if we apply
\( \U_{\A^\infty} \) to a hook-valued tableau \( H \), then we obtain
a pair \( (S,E) \) of a set-valued tableau \( S \) and a
column-flagged increasing tableau \( E \). However, the exact
description of the image \( \U_{\A^\infty}(\HVT(\lambda)) \) is
unknown.

\begin{problem}
  Find a combinatorial proof of the expansion \eqref{eq:4}.
\end{problem}

There is a parallel story on dual Grothendieck polynomials; see
\cite{LP.2007, GGL2016} and references therein. Hwang et
al.~\cite[Definition~1.2]{HJKSS1} introduced refined dual canonical
stable Grothendieck polynomials
\( g_\lambda(\bm{x}, \bm{\alpha},\bm{\beta}) \), which are dual to
\( G_\lambda(\bm{x}, \bm{\alpha},\bm{\beta}) \) with respect to the
Hall inner product. They found two combinatorial descriptions for
\( g_\lambda(\bm{x}, \bm{\alpha},\bm{\beta}) \).

\begin{problem}
  Find a combinatorial proof
  of the equivalence of the two models
  for \( g_\lambda(\bm{x}, \bm{\alpha},\bm{\beta}) \)
  in \cite[Theorem~1.6 and Corollary 4.6]{HJKSS1}.
\end{problem}

\bibliographystyle{alpha}
\bibliography{main.bib}

\end{document}